\documentclass[11pt,a4paper,reqno]{amsart}
\usepackage[utf8]{inputenc}
\usepackage{stmaryrd}
\usepackage[english]{babel}
\usepackage{pst-grad} 
\usepackage{pst-plot} 
\usepackage{pstricks}
\usepackage{amsmath,amssymb,mathrsfs,amsthm, mathtools}
\usepackage{tikz} 
\usepackage{mathcomp,wasysym}  
\usepackage{graphicx}  
\usepackage[all]{xy} \xyoption{arc} \xyoption{color}
\usepackage{epsfig}
\usepackage{float}

\usepackage{cite}
\usepackage[a4paper,
left=2.5cm, right=2.5cm,
top=3cm, bottom=3cm]{geometry}

\newcommand{\dx}{\textnormal{d}{x}}

\newcommand{\norm}[1]{\left\| #1 \right\|}

\newcommand{\R}{\mathbb R}

\newcommand{\T}{\mathbb T}

\normalsize
\normalsize
\newcommand{\abs}[1]{\left\lvert #1 \right\rvert}

\numberwithin{equation}{section}

\usepackage[bookmarks,colorlinks=true, citecolor=blue]{hyperref}

\newtheorem{satz}{Proposition}[section]
\newtheorem{remark}[satz]{Remark} 
\newtheorem{theorem}{Theorem}

\newtheorem{lemma}[satz]{Lemma}

\title{Asymptotic models for the evolution of a circular biofilm}

\author[D. Alonso-Or\'{a}n]{Diego Alonso-Or\'{a}n}
\email{dalonsoo@ull.edu.es}
\address{Departamento de An\'{a}lisis Matem\'{a}tico y Instituto de Matem\'{a}ticas y Aplicaciones (IMAULL), Universidad de La Laguna C/. Astrof\'{i}sico Francisco S\'{a}nchez s/n, 38200 - La Laguna, Spain.}

\author[A. Carpio]{Ana Carpio}
\email{ana\_carpio@mat.ucm.es}
\address{Departamento de An\'{a}lisis Matem\'{a}tico y  Matem\'{a}tica  Aplicacada, Universidad Complutense de Madrid, Plaza de Ciencias 3, 28040 - Madrid, Spain}

\author[R. Granero-Belinch\'{o}n]{Rafael Granero-Belinch\'{o}n}
\email{rafael.granero@unican.es}
\address{Departamento  de  Matem\'aticas,  Estad\'istica  y  Computaci\'on,  Universidad  de Cantabria,  Avda.  Los  Castros  s/n,  Santander,  Spain.}

\begin{document}
\begin{abstract}
We study a class of lubrication-type equations modeling the spread of thin poroelastic biofilms at air/agar interfaces. Starting from a biofilm slab model, we perform a formal multi-scale asymptotic expansion to derive a reduced nonlinear evolution equation with time-dependent coefficients. First, we establish a local-in-time existence result as well as a continuation criteria. Moreover, under suitable assumptions on the radius of the biolfilm, we show the existence of global in time solutions. We conclude with some numerical simulations illustrating the emergence of instabilities and potential singularities in finite time. 
\end{abstract}

\subjclass[2010]{}


\maketitle

\allowdisplaybreaks

\section{Introduction}

In Nature, bacteria often gather  on moist surfaces to form communities called biofilms \cite{flemming}.
In the biofilm habitat, bacteria are immersed in a self-produced polymeric matrix that shelters them from external aggressions such as antibiotics, chemicals and flows. Biofilm dynamics and structure change noticeably with the environmental conditions. While biofilms growing in flows tend to form filamentary structures \cite{stone}, biofilms growing on interfaces with air spread on the surface displaying wrinkled patterns \cite{seminara}. There is a large literature on biofilm  modelling at different scales, ranging from microscopic agent based computational  approaches \cite{allen, ccp, espeso, laspidou, storck} to continuous models of macroscopic behavior \cite{cogan, kapellos}. 
We focus here on continuous mathematical models for biofilm spread on  air/agar interfaces \cite{seminara}.  A typical approach considers the biofilm as a mixture of two phases: a liquid phase
(water, dissolved chemicals and so on) modeled as an inviscid fluid and a biomass phase modeled either as a viscous fluid \cite{seminara} or an elastic material \cite{amm}, depending on whether the biofilm displays more fluid-like or solid-like behavior. The resulting system is formulated in the biofilm 
occupied region, which moves with time. 
For thin biofilms, one can implement a lubrication type approximation to obtain an equation for the motion of the biofilm boundary, coupled to the equations of conservation of mass and momentum inside the biofilm \cite{seminara, entropy}. 
So far, well-posedness results for these models assume the  motion of the biofilm boundary known and establish existence and stability estimates for solutions in quasi-static  regimes \cite{na} or for specific types of boundary deformations with time \cite{amm}. Attempts have been made to analyze the motion of the boundary.
Under additional conditions on the geometry and the parametric regime, a single equation for the evolution of  the biofilm boundary follows from explicit asymptotic profiles for volume fractions and velocities inside the biofilm  obtained by a perturbation analysis \cite{seminara, entropy}.
For dense circular biofilms in which the biomass phase is considered elastic, the height $h(r,t)$ of a biofilm of radius $R(t)$ evolves according to the dimensionless
equation \cite{entropy}
\begin{equation}\label{biofilm_orig}
h_t - \frac{KR}{rR_0}(rh_{rt}h^3)_r-\frac{3KR}{2rR_0}(rh_rh^2h_t)_r 
- \frac{KR_t}{rR_0}(rh_rh^3)_r - {\delta^2 \over r} (r h^2 h_r)_r
= h. \end{equation}
The parameter $K$ encodes the rheological properties
of the different phases. Assuming $\delta \sim{h_0 \over R_0}  \ll 1$, that is, a thin film whose average height $h_0$ remains much smaller that its average radius $R_0$ during the computation time, we can neglect the
last term. Then, the change $h = e^t \tilde h$ leads to
\begin{equation}\label{biofilm}
h_t-\frac{KR}{rR_0}e^{3t}(rh_{rt}h^3)_r-\frac{3KR}{2rR_0}e^{3t}(rh_rh^2h_t)_r-(\frac{5KR}{2rR_0}
+\frac{KR_t}{rR_0})e^{3t}(rh_rh^3)_r = 0,
\end{equation}
dropping the tilde symbol for ease of notation.
If we work with biofilm slabs instead of circular patches, equation (\ref{biofilm_orig}) is replaced by
\begin{equation}\label{biofilm_orig_x}
h_t - \frac{KR}{R_0}(h_{xt}h^3)_x-\frac{3KR}{2R_0}(h_xh^2h_t)_x 
- \frac{KR_t}{R_0}(h_xh^3)_x - \delta^2 (h^2 h_x)_x
= h, 
\end{equation}
or 
\begin{equation}\label{biofilm_x}
h_t-\frac{KR}{R_0}e^{3t}(h_{xt}h^3)_x-\frac{3KR}{2R_0}e^{3t}(h_x h^2h_t)_x-(\frac{5KR}{2R_0}
+\frac{KR_t}{R_0})e^{3t}(h_xh^3)_x - \delta^2 e^{2t}(h^2 h_x)_x= 0.
\end{equation}
 In principle, we could impose a boundary condition at the moving border of the biofilm occupied region $r = R(t)$ or $|x| = R(t)$, considered as triple point where three boundaries meet: air, biofilm and surface. Alternatively, one can follow  \cite{degennes} and insert `precursor layers' at the edges, that is, we assume that the spreading biofilm is always surrounded by a region at which $h = \eta$, for a small positive constant $\eta$.

Lubrication type approximations have been long used to derive equations for the moving boundaries of a variety of thin films, ameneable to analysis and numerical approximation. For fluid-like biofilms in 
which the biomass phase is considered a viscous fluid, the height $h(r,t)$ of a biofilm of radius $R(t)$ is described by simpler equations \cite{seminara}:
\begin{equation}\label{biofilmv_orig}
h_t - \frac{KR}{rR_0}(rh_rh^3)_r - {\delta^2 \over r} (r h^2 h_r)_r = h, 
\end{equation}
which can be approximated for small $\delta$ by expressions of the form 
\begin{equation}\label{biofilmv}
h_t- e^{3t} \frac{KR}{rR_0}(rh_rh^3)_r = 0.
\end{equation}
These are porous media type equations, easier to handle, except for the fact that the radius or length
$R(t)$ is usually unknown too.
Seeking self-similar solutions we find a combination of $h$ and $R$ that solves (\ref{biofilmv_orig})
with $\delta=0$:
\begin{eqnarray} \label{selfsimilarv}
h =   {e^{t}  \over (R/R_0)^2} \left( 1-{3 \over 2} {r^2 \over R^2}\right)^{1/3},  \quad
R = R_0 \left(1 + {7 \over 3}  K (e^{3t} -1)\right)^{1/7}.
\end{eqnarray}
Introducing expressions of this form (after replacing ${7 \over 3}$ by ${35 \over 6}$) in (\ref{biofilm_orig}), the remainder decays to zero as $t$ grows, see \cite{entropy}. Thus, these expressions provide asymptotic profiles  to which some initial data evolve under (\ref{biofilm}) for large times, as shown by numerical solutions. However, they are not solutions.

Classical lubrication models for material thin fluid films of different rheology (newtonian, non-newtonian, surfactants) read
\begin{eqnarray}
h_t + \tau h h_x + \sigma (h^3 h_{xxx})_x =0 \quad \mbox{\rm for constant shear stress and surface tension}, \\
h_t - \rho g (h^3 h_x)_x + \sigma (h^3 h_{xxx})_x =0 \quad \mbox{\rm for constant surface tension and gravity},
\label{instability} 
\end{eqnarray}
and comprise also variants  including thermal effects, evaporation, condensation, van der Waals forces, falling motion, and so on. In contrast with the biofilm models, the coefficients that appear in the resulting equations are always time independent and the surface  covered is fixed: the only unknown is the biofilm height $h$. Equation (\ref{instability}) is well known to exhibit rupture phenomena (Raleigh-Taylor instability). Linear stability analyses of perturbed constant profiles  show that local finite time vanishing (rupture) of the film is possible for some parameter ranges and staring conditions (wavy profiles, perturbed constants). Well-posedness results for these models are established for instance in \cite{bertozzi}.

Equation (\ref{biofilm_x})  presents novel features that make its study particularly challenging: the presence of spatial derivatives of the time derivatives and the presence of time dependent coefficients (in principle part of the solution, as illustrated by the computation of the self-similar profiles). Here, we will focus on constructing solutions for $\delta=0$ and given functions $R(t)$. Although equations involving operators of the form \( u_t - A u_t \) have been studied previously~\cite{ambrose}, existing results typically assume time-independent coefficients, and therefore do not directly apply to our setting. Instead, we will try to solve the initial value problem taking perturbed constants as initial data. Notice that  (\ref{biofilm}) and (\ref{biofilm_x}) admit any constant as a solution. However, in which functional spaces these problems may be well posed is a priori unclear. \medskip

In this framework, the purpose of this article is three-fold. First, by means of a formal multi-scale asymptotic expansion \cite{alonso2024, cheng2019, granero2019} we reduce equation \eqref{biofilm_x} into a cascade of linear partial differential equations which can be closed up to some order of precision. More precisely, the reduced equation
takes the form
\begin{equation}\label{model2:intro}
	f_t-\mathcal{L}_tf=\mathcal{Q}_t\frac{3KR}{2R_0}e^{3t}(f_x\mathcal{L}_tf)_x+3\mathcal{Q}_t\frac{KR}{R_0}e^{3t}(\mathcal{L}_tf_xf)_x
	+3\mathcal{Q}_t(\frac{5KR}{2R_0}+\frac{KR_t}{R_0})e^{3t}(f_{x}f)_x,
\end{equation}
which we call the f-model, where the time dependent operators $\mathcal{Q}_t$ and $\mathcal{L}_t$ are defined in terms of Fourier multipliers.  Asymptotic and multi-scale expansion methods provide a formal strategy to derive reduced models that approximate the full dynamics in specific parameter regimes. These asymptotic systems retain the leading-order behavior while offering significant simplifications, making them more amenable to mathematical analysis and numerical simulation. Classical applications include the asymptotic derivation of the Saint-Venant, Green--Naghdi, and Boussinesq equations from the incompressible Euler system, to model the evolution of shallow water flows, \cite{lannes}. Beyond shallow water theory, asymptotic expansions have also been employed to derive models for the rotational shear and magnetic confinement in the solar tachocline~\cite{alonso2021}, as well as reduced fluid models for the evolution of cold and magnetized plasmas; see, e.g.,~\cite{alonso2024}. \medskip

The second purpose of this work is the study of several analytical properties of \eqref{model2:intro}. Specifically, the results shown in this paper can be summarized as follows:
\begin{itemize}
\item For $R(t) = e^{\alpha t}$, $\alpha \in \R$, we establish the local-in-time existence and uniqueness of solutions to \eqref{model2:intro} for initial data in suitable Sobolev spaces (Theorem \ref{result:th}). Moreover, we also establish a blow-up criteria to characterize the presence of possible singularities (Theorem \ref{cont:criterion}).
\item For $\alpha>-2$, we can show the existence of global-in-time existence of solutions to \eqref{model2:intro} in Sobolev spaces (Theorem \ref{theorem:global}).
\end{itemize}\medskip

To conclude, we perform numerical simulations to illustrate the dynamical behavior of the models introduced in this work. We explore the formation of self-similar profiles and the interaction between vertical growth and lateral spreading under different choices of the scaling factor $R(t)$. The numerical tests highlight the stability and instability mechanisms predicted by the theory, including the emergence of flat regions, wavy structures, and potential singular behaviors. In particular, we investigate the autonomous $f$-model and support, through simulations, the linear instability of constant solutions when $\delta = 0$ and possible finite time blow-up.

\subsection*{Plan of the paper} 
The structure of the paper is as follows. In Section~\ref{sec:der} ,we present the asymptotic derivation of the $f$-model studied in this paper by means of an multi-scale expansion. Section~\ref{notation:aux} is devoted to fixing the notation and presenting several auxiliary lemmas concerning the boundedness properties the previous Fourier multiplier operators. In Section~\ref{sec:3}, we establish the local-in-time existence and uniqueness of solutions to the \( f \)-model~\eqref{model2} for suitable Sobolev initial data, and we provide blow-up criteria characterizing the breakdown of smooth solutions.  In Section~\ref{sec:4}, we extend the local existence theory for the \( f \)-model~\eqref{model2} by proving global-in-time result for certain values of the parameter \( \alpha \). To conclude in Section~\ref{sec:5} we provide some numerical simulations illustrating the emergence of instabilities and potential singularities in finite time.

\subsection*{Acknowledgments}
 D.A-O has been partially supported by RYC2023-045563-I (MICIU/AEI/1\\ 
 0.13039/501100011033 and FSE+) and the project PID2023-148028NB-100. D.A-O and R. G-B are also supported by the project “An\'alisis Matem\'atico Aplicado y Ecuaciones Diferenciales” Grant PID2022- 141187NB-I00 funded by MCIN/ AEI and acronym “AMAED”. 
 A.C. is supported by the projects “Métodos y modelos matemáticos para aplicaciones biomédicas”, Grant PID2020-112796RB-C21with acronym “MATHBIOMED”, and “Movimiento celular colectivo y matemáticas aplicadas a la biomedicina”, Grant PID2024-155528OB-C21 with acronym “MATHCELL”, 
funded by MCIN/ AEI.

\section{Derivation of the f-model for the biolfim equation}\label{sec:der}
We begin by deriving reduced equations describing the evolution of perturbations around constant solutions of~\eqref{biofilm_x} with \( \delta = 0 \). Let us assume that 
$$
h(x,t)=1+\varepsilon g(x,t)
$$
and $0<\varepsilon\ll 1$. Then, retaining $O(\varepsilon^2)$ terms, \eqref{biofilm_x} reads
\begin{equation}\label{biofilm3}
g_t-\frac{KR}{R_0}e^{3t}(g_{xt}(1+3\varepsilon g))_x-\frac{3KR}{2R_0}e^{3t}(g_xg_t)_x-(\frac{5KR}{2R_0}+\frac{KR_t}{R_0})e^{3t}(g_x(1+3\varepsilon g))_x=0.
\end{equation}
We introduce the asymptotic expansion
\begin{equation*}
	g(z,t) = \sum_{\ell=0}^\infty \varepsilon^\ell g^{(\ell)}(z,t),
\end{equation*}
where \( \varepsilon \ll 1 \) is a small parameter. Substituting this Ansatz into the \eqref{biofilm3} we obtain a hierarchy (or cascade) of linear partial differential equations for the successive profiles \( g^{(\ell)} \), namely
\begin{multline}\label{biofilm4}
g^{(\ell)}_t-\frac{KR}{R_0}e^{3t}g^{(\ell)}_{xxt}-(\frac{5KR}{2R_0}+\frac{KR_t}{R_0})e^{3t}g^{(\ell)}_{xx}=\frac{3KR}{2R_0}e^{3t}\left(\sum_{j=0}^{\ell-1}g^{(j)}_xg^{(\ell-j-1)}_t\right)_x+3\frac{KR}{R_0}e^{3t}\left(\sum_{j=0}^{\ell-1}g^{(j)}_{xt}g^{(\ell-j-1)}\right)_x\\
 +3(\frac{5KR}{2R_0}+\frac{KR_t}{R_0})e^{3t}\left(\sum_{j=0}^{\ell-1}g^{(j)}_{x}g^{(\ell-j-1)}\right)_x=0.
\end{multline}
In particular, $g^{(0)}(x,t)$ verifies that
$$
g^{(0)}_t-\frac{KR}{R_0}e^{3t}g^{(0)}_{xxt}-(\frac{5KR}{3R_0}+\frac{KR_t}{R_0})e^{3t}g^{(0)}_{xx}=0.
$$
We observe that we can equivalently write
$$
g^{(0)}_t=\mathcal{L}_tg^{(0)},
$$
where the operator $\mathcal{L}_t$ is defined as the following Fourier multiplier
$$
\widehat{\mathcal{L}_t G}(n)=-\frac{(\frac{5KR}{2R_0}+\frac{KR_t}{R_0})e^{3t}n^2}{1+\frac{KR}{R_0}e^{3t}n^2}\widehat{G}(n).
$$
Similarly, we infer that $g^{(1)}(x,t)$  satisfies
\begin{multline*}
g^{(1)}_t-\frac{KR}{R_0}e^{3t}g^{(1)}_{xxt}-(\frac{5KR}{2R_0}+\frac{KR_t}{R_0})e^{3t}g^{(1)}_{xx}\\=\frac{3KR}{2R_0}e^{3t}(g^{(0)}_x\mathcal{L}_tg^{(0)})_x+3\frac{KR}{R_0}e^{3t}(\mathcal{L}_tg^{(0)}_xg^{(0)})_x
+3(\frac{5KR}{2R_0}+\frac{KR_t}{R_0})e^{3t}(g^{(0)}_{x}g^{(0)})_x.
\end{multline*}
Hence, defining 
$$
f(x,t)=g^{(0)}(x,t)+\varepsilon g^{(1)}(x,t),
$$
we conclude that, up to $O(\varepsilon^2)$ corrections, $f(x,t)$ solves 
\begin{multline}\label{model}
f_t-\frac{KR}{R_0}e^{3t}f_{xxt}-(\frac{5KR}{2R_0}+\frac{KR_t}{R_0})e^{3t}f_{xx}\\=\frac{3KR}{2R_0}e^{3t}(f_x\mathcal{L}_tf)_x+3\frac{KR}{R_0}e^{3t}(\mathcal{L}_tf_xf)_x
+3(\frac{5KR}{2R_0}+\frac{KR_t}{R_0})e^{3t}(f_{x}f)_x.
\end{multline}
Moreover, defining the operator $\mathcal{Q}_{t}$ via the Fourier multiplier
$$
\widehat{\mathcal{Q}_t G}(n)=\frac{1}{1+\frac{KR}{R_0}e^{3t}n^2}\widehat{G}(n),
$$
and rearranging terms, equation~\eqref{model} can be equivalently written as
\begin{equation}\label{model2}
	f_t-\mathcal{L}_tf=\mathcal{Q}_t\frac{3KR}{2R_0}e^{3t}(f_x\mathcal{L}_tf)_x+3\mathcal{Q}_t\frac{KR}{R_0}e^{3t}(\mathcal{L}_tf_xf)_x
	+3\mathcal{Q}_t(\frac{5KR}{2R_0}+\frac{KR_t}{R_0})e^{3t}(f_{x}f)_x.
\end{equation}
We refer to equation~\eqref{model2} as the non-autonomous \( f \)-model associated with the biofilm equation~\eqref{biofilm_x}. A further simplified autonomous model can be obtained if we prescribed 
$$
\frac{K}{R_0}=1\text{ and }R(t)=e^{-3t}.
$$
Indeed, with this particular choice, we obtain that $\mathcal{L}_t = \mathcal{L}$ and $\mathcal{Q}_t = \mathcal{Q}$, where the operators $\mathcal{L}$ and $\mathcal{Q}$ are defined in Fourier space by
\[
\widehat{\mathcal{L} G}(n) = \frac{n^2}{2(1 + n^2)} \widehat{G}(n), \qquad \widehat{\mathcal{Q} G}(n) = \frac{1}{1 + n^2} \widehat{G}(n).
\]
Accordingly, equations~(\ref{model}) and~(\ref{model2}) reduce to
\begin{equation}\label{model:aut}
f_t-f_{xxt}+\frac{1}{2}f_{xx}=\frac{3}{2}(f_x\mathcal{L}f)_x+3(\mathcal{L}f_xf)_x
-\frac{3}{2}(f_{x}f)_x,
\end{equation}
and
\begin{equation}\label{model2:aut}
f_t-\mathcal{L}f=\frac{3}{2}\mathcal{Q}(f_x\mathcal{L}f)_x+3\mathcal{Q}(\mathcal{L}f_xf)_x
+\frac{3}{2}\mathcal{L}(f^2).
\end{equation}
To maintain consistency in terminology, we shall refer to equation~\eqref{model2:aut} as the autonomous \( f \)-model.

\begin{remark}
For the sake of simplifying, to derive \eqref{model2:aut}, we have assumed that
$\frac{K}{R_0} =1$. This normalization is introduced solely for notational convenience and entails no loss of generality, as the parameter \(K/R_0\) acts as an overall multiplicative constant in the equation. 
\end{remark}

\begin{remark}
We remark that the derivation of the asymptotic models \eqref{model2} and \eqref{model2:aut} presented here is purely formal. Whether the non-autonomous \( f \)-model and the autonomous \( f \)-model constitute fully justified asymptotic models remains an open question. Addressing this question would require a rigorous justification, akin to those developed in the context of water wave theory, cf. \cite{lannes}.
\end{remark}

\begin{remark}
In order to get a better understanding of equations \eqref{model2} and \eqref{model2:aut} let us compute the linear dispersion relation of \eqref{model2}. To that purpose, we consider $f=1+g$
and neglecting terms of order $\mathcal{O}(g^2)$, we find the linear equation
\begin{equation}\label{model2linear}
	g_t-\mathcal{L}_tg=3\mathcal{Q}_t\frac{KR}{R_0}e^{3t}\mathcal{L}_tg_{xx}
	+3\mathcal{Q}_t(\frac{5KR}{2R_0}+\frac{KR_t}{R_0})e^{3t}g_{xx}.
\end{equation}
Using the definition of $\mathcal{L}_t$ and $\mathcal{Q}_t$, and grouping similar terms, equation \eqref{model2linear} can be simplified as follows
\begin{equation}\label{model2linear2}
	g_t=3\mathcal{Q}_t\frac{KR}{R_0}e^{3t}\mathcal{L}_tg_{xx}
	+4\mathcal{L}_tg.
\end{equation}
Similarly, for \eqref{model2:aut}, we find the linear problem
\begin{equation}\label{model2:linearaut}
g_t=3\mathcal{L}^2g+4\mathcal{L}g.
\end{equation}
Taking Fourier transform in the space variable, we find that
\[
\widehat{g}_t(n) = \lambda(n)\, \widehat{g}(n), \quad \lambda(n) = \frac{3 n^4}{4(1 + n^2)^2} + \frac{2 n^2}{1 + n^2},
\]
leading to distinct exponential growth rates across Fourier modes:
\[
\hat{g}(n,t) = \hat{g}(n,0)\, e^{\lambda(n)t}.
\]
This frequency-dependent amplification deforms the solution over time. In particular, the absence of uniform decay or smoothing across modes confirms that the equation is not parabolic.

\end{remark}

\section{Notation and auxiliary results}\label{notation:aux}

In this subsection, we collect the different notations that will be used throughout the manuscript as well as some auxiliary results regarding the differential operators involved in the equation. 

\subsubsection*{Functional spaces and inequalities} Throughout the paper $C= C(\cdot)$ will denote a positive constant that may depend on fixed parameters and can vary from line to line. Let $\T=(\mathbb{R}/2\pi \mathbb{Z})$ be the one-dimensional torus. Then, for $1\leq p\leq\infty$, we write $L^{p}=L^{p}(\T)$ to denote the normed space of $L^{p}$-functions on $\T$ with $||\cdot ||_{p}$ as the associated norm and let  $\widehat{f}$ denote the Fourier transform of $f$. Then,  for $s\in\T$, the inhomogeneous Sobolev space $H^{s}=H^s(\T)$ is defined as
\begin{align*}
	H^s(\T)\triangleq\left\{f\in L^2(\T):\|f\|_{H^s(\T)}^2=\displaystyle\sum_{n\in\mathbb{Z}}(1+|n|^{s})^{2}|\widehat{f}(n)|^{2}<+\infty\right\}.
\end{align*}
Moreover, we have that
\[ \norm{f}^{2}_{H^s}=\norm{f}_{L^2}^2+\norm{f}^{2}_{\dot{H}^s}, \quad \norm{f}_{\dot{H}^s}^{2}=\displaystyle\sum_{n\in\mathbb{Z}\setminus\{0\}}|n|^{2s}
|\widehat{f}(n)|^{2}.  \]

Next, we collect some useful classical inequalities that will be repeatedly use throughout the manuscript, cf. \cite{majda2001}:

\begin{lemma}\label{ine:Sobolev}
	Let \( f : \mathbb{T} \to \mathbb{R} \) be a periodic function with zero mean. Then the following estimates hold:
	
	\begin{itemize}
		\item[(i)]  If \( f \in L^\infty(\mathbb{T}) \) and \(  f_{xx} \in L^2(\mathbb{T}) \), then 
		\[
		\|f\|_{L^4(\mathbb{T})} \leq \sqrt{3} \|f\|_{L^\infty(\mathbb{T})}^{1/2} \|f_{xx}\|_{L^2(\mathbb{T})}^{1/2}.
		\]
		
		\item[(ii)]  If \( s > \frac{1}{2} \) and \( f \in H^s(\mathbb{T}) \), then
		\[
		\|f\|_{L^\infty(\mathbb{T})} \leq C_s \|f\|_{H^s(\mathbb{T})},
		\]
		where \( C_s > 0 \) depends only on \( s \).
	\end{itemize}
\end{lemma}

\subsubsection*{Estimates for the differential operators} In what follows, we collect some important properties of the differential operators $\mathcal{Q}_{t}$ and $\mathcal{L}_{t}$. Let us recall that for $G\in L^{2}(\T)$
the operators $\mathcal{Q}_{t}$ and $\mathcal{L}_{t}$ are defined via their Fourier multipliers as follows
\begin{align}
	\widehat{\mathcal{Q}_t G}(n)&=\frac{1}{1+\frac{KR}{R_0}e^{3t}n^2}\widehat{G}(n),  \label{def:Q} \\
	\widehat{\mathcal{L}_t G}(n)&=-\frac{(\frac{5KR}{2R_0}+\frac{KR_t}{R_0})e^{3t}n^2}{1+\frac{KR}{R_0}e^{3t}n^2}\widehat{G}(n), \quad \label{def:L} 
\end{align}

Let us first derive $L^p$ - estimate for the non-autonomous Helmholtz operator $\mathcal{Q}_{t}$ and the operator $\mathcal{L}_{t}$. 
\begin{lemma}\label{Lemma:Lp}
Let $R(t)=e^{\alpha t}$ with $\alpha\in\mathbb{R}$. Let $\mathcal{Q}_{t}$, $\mathcal{L}_{t}$ be the differential operators given in \eqref{def:Q} and \eqref{def:L}, respectively. 
	Then the following properties hold:
	\begin{enumerate}
		\item[\textnormal{(i)}]  For all smooth periodic functions \( f \), the following identity holds:
		\[
		\mathcal{L}_t f = \left( \frac{5}{2} + \alpha \right) \left( \mathcal{Q}_t f - f \right).
		\]
		
		\item[\textnormal{(ii)}] For all \( 1 \leq p \leq \infty \), the operator \( \mathcal{Q}_t \) extends to a bounded linear operator on \( L^p(\mathbb{T}) \), and there exists a constant \( C > 0 \) independent of \( f \) and \( t \), such that
		\[
		\| \mathcal{Q}_t f \|_{L^p(\mathbb{T})} \leq C \| f \|_{L^p(\mathbb{T})}.
		\]
		
		\item[\textnormal{(iii)}] For all \( 1 \leq p \leq \infty \), the operator \( \mathcal{L}_t \) also extends to a bounded linear operator on \( L^p(\mathbb{T}) \). More precisely, there exists a constant \( C> 0 \) such that
		\[
		\| \mathcal{L}_t f \|_{L^p(\mathbb{T})} \leq C \| f \|_{L^p(\mathbb{T})}, \quad \text{for all } f \in L^p(\mathbb{T}).
		\]
	\end{enumerate}
	
	\begin{proof} [Proof of Lemma \ref{Lemma:Lp}]
		
		\medskip

\noindent (i) The identity follows direcly by noting that
				\begin{equation*}
				\widehat{\mathcal{L}_t f}(n)
				= -\left( \frac{5}{2} + \alpha \right)\frac{\left[-1+\left(1+ \frac{K}{R_0} e^{(\alpha + 3)t} n^2\right) \right]}{1 + \frac{K}{R_0} e^{(\alpha + 3)t} n^2}  \widehat{f}(n)
				= \left( \frac{5}{2} + \alpha \right) \left( \widehat{\mathcal{Q}_t f}(n) - \widehat{f}(n) \right).
			\end{equation*}
			
\noindent (ii)	For each fixed \( t \), define \( \lambda(t)^2 := \frac{K}{R_0} e^{(\alpha + 3)t} \). Then \( \mathcal{Q}_t = (\mathrm{Id} - \lambda(t)^2 \partial_{xx})^{-1} \), the inverse of a second-order elliptic operator with periodic boundary conditions. This inverse operator admits the representation
			\[
			\mathcal{Q}_t f = \mathsf{G}_t \ast f,
			\]
			where \( \mathsf{G}_t \) is the Green's function of the operator \( \mathrm{Id} - \lambda(t)^2 \partial_{xx} \) on the torus \( \mathbb{T} \). The Green's function \( \mathsf{G}_t \) is explicitly given by
			\[
			\mathsf{G}_t(x) = \frac{1}{2\lambda(t) \sinh(\pi/\lambda(t))} \cosh\left( \frac{x - 2\pi \left\lfloor \frac{x}{2\pi} \right\rfloor - \pi}{\lambda(t)} \right), \quad x \in \mathbb{T}.
			\]
			It is smooth, positive, even, and satisfies
			\[
			\int_{\mathbb{T}} \mathsf{G}_t(x) \, dx = 1,
			\]
			for all \( t \in \mathbb{R} \). Therefore, applying Young's inequality for convolutions on \( \mathbb{T} \), we obtain
			\[
			\| \mathcal{Q}_t f \|_{L^p(\mathbb{T})} = \| \mathsf{G}_t \ast f \|_{L^p(\mathbb{T})} \leq \| \mathsf{G}_t \|_{L^1(\mathbb{T})} \| f \|_{L^p(\mathbb{T})} = \| f \|_{L^p(\mathbb{T})},
			\]
			for every \( 1 \leq p \leq \infty \).

\noindent (iii)	 Using the identity from (i), we write:
			\[
			\mathcal{L}_t f = \left( \frac{5}{2} +\alpha \right) \left( \mathcal{Q}_t f - f \right).
			\]
			Taking norms in \( L^p(\mathbb{T}) \) and applying the triangle inequality and the property (ii), we obtain:
			\[
			\| \mathcal{L}_t f \|_{L^p} \leq \left|  \left( \frac{5}{2} +\alpha \right) \right| \left( \| \mathcal{Q}_t f \|_{L^p} + \| f \|_{L^p} \right)
			\leq C  \| f \|_{L^p},
			\]
			showing that \( \mathcal{L}_t \) is bounded on \( L^p(\mathbb{T}) \).
	\end{proof}
\end{lemma}

Next, we will provide the $H^{s}$ - bounds for the non-autonomous Helmholtz operator $\mathcal{Q}_{t}$ and the operator $\mathcal{L}_{t}$.

\begin{lemma}\label{Lemma:Sobolev}
Let $R(t)=e^{\alpha t}$ with $\alpha\in\mathbb{R}$. Let $\mathcal{Q}_{t}$, $\mathcal{L}_{t}$ be the differential operators given in \eqref{def:Q} and \eqref{def:L}, respectively.  Then the following properties hold:
	\begin{itemize}
		\item[(i)] The operator \( \mathcal{Q}_{t} \) is a smoothing operator of order \( -2 \). More precisely, there exists a constant \( C > 0 \), independent of \( t \), such that
		\[
		\| \mathcal{Q}_{t} f \|_{H^{s+2}(\mathbb{T})} \leq C \| f \|_{H^s(\mathbb{T})}, \quad \text{for all } f \in H^s(\mathbb{T}), \ s \in \mathbb{R}.
		\]
		\item[(ii)] The operator \( \mathcal{L}_{t} \) is of order zero. More precisely, there exists a constant \( C > 0 \), independent of \( t \), such that
		\[
		\| \mathcal{L}_{t} f \|_{H^s(\mathbb{T})} \leq C \| f \|_{H^s(\mathbb{T})}, \quad \text{for all } f \in H^s(\mathbb{T}), \ s \in \mathbb{R}.
		\]
	\end{itemize}
\end{lemma}

\begin{proof} [Proof of Lemma \ref{Lemma:Sobolev}]
	We analyze the symbols of the operators \( \mathcal{Q}_t \) and \( \mathcal{L}_t \) directly in Fourier space.
	
	\medskip
	
	\noindent (i) The symbol of \( \mathcal{Q}_t \) is
	\[
	m_{\mathcal{Q}_t}(n) = \frac{1}{1 + \frac{K}{R_0} e^{(\alpha + 3)t} n^2}.
	\]
	To estimate the \( H^{s+2} \)-norm of \( \mathcal{Q}_t f \), we compute
	\[
	\| \mathcal{Q}_t f \|_{H^{s+2}(\mathbb{T})}^2 = \sum_{n \in \mathbb{Z}} (1 + n^2)^{s+2} \left| m_{\mathcal{Q}_t}(n) \widehat{f}(n) \right|^2.
	\]
	Observe that for all \( n \in \mathbb{Z} \),
	\[
	(1 + n^2)^{s+2} \left| \frac{1}{1 + \frac{K}{R_0} e^{(\alpha + 3)t} n^2} \right|^2
	= (1 + n^2)^s (1 + n^2)^2  \left( \frac{1}{1 + \frac{K}{R_0} e^{(\alpha + 3)t} n^2} \right)^2.
	\]
	For \( |n| \geq 1 \), the factor
	\[
	\frac{(1 + n^2)^2}{\left( 1 + \frac{K}{R_0} e^{(\alpha + 3)t} n^2 \right)^2},
	\]
	is uniformly bounded in \( t \) and \( n \), because for large \( n \), it behaves like \( \left( \frac{1}{\frac{K}{R_0} e^{(\alpha + 3)t}} \right)^2 \), and for small \( n \), all terms are bounded.
	
	Thus,
	\[
	(1 + n^2)^{s+2} \left| m_{\mathcal{Q}_t}(n) \right|^2 \leq C (1 + n^2)^s
	\]
	with a constant \( C > 0 \) independent of \( t \). Therefore,
	\[
	\| \mathcal{Q}_t f \|_{H^{s+2}}^2 \leq C \sum_{n \in \mathbb{Z}} (1 + n^2)^s |\widehat{f}(n)|^2 = C \| f \|_{H^s}^2.
	\]
	
	\medskip
	
	\noindent (ii) The symbol of \( \mathcal{L}_t \) is
	\[
	m_{\mathcal{L}_t}(n) = -\frac{\left( \frac{5}{2} +  \alpha \right) \frac{K}{R_{0}}e^{(\alpha + 3)t} n^2}{1 + \frac{K}{R_0} e^{(\alpha + 3)t} n^2}.
	\]
	Note that for all \( n \in \mathbb{Z} \),
	\[
	|m_{\mathcal{L}_t}(n)| \leq \left(  \frac{5}{2} +  |\alpha| \right)  \frac{\frac{K}{R_0}e^{(\alpha + 3)t} n^2}{1 + \frac{K}{R_0} e^{(\alpha + 3)t} n^2}
	\leq \left( \frac{5}{2} +  |\alpha|\right)=C
	\]
	which is independent of \( n \) and \( t \). Hence,
	\[
	(1 + n^2)^s |m_{\mathcal{L}_t}(n) \widehat{f}(n)|^2 \leq C^2 (1 + n^2)^s |\widehat{f}(n)|^2,
	\]
	and summing over \( n \in \mathbb{Z} \), we obtain
	\[
	\| \mathcal{L}_t f \|_{H^s}^2 \leq C^2 \| f \|_{H^s}^2.
	\]
\end{proof}

\begin{remark}\label{eleccion:R}
	Throughout this section, all results have been established under the particular assumption \( R(t) = e^{\alpha t} \), with \( \alpha \in \mathbb{R} \). This specific form is essential for the validity of the identity stated in (i) in Lemma \ref{Lemma:Lp}, which plays a central role in deriving several key a priori estimates in Sobolev spaces, cf. Section \ref{sec:3}. While it is conceivable that similar results could be obtained for more general choices of \( R(t) \), the exponential form is not only mathematically convenient but also physically motivated. In particular, it corresponds to a class of flows with exponential dilation, which are of direct relevance in the physical scenarios we aim to model. For this reason, we have focused our analysis on this case.
\end{remark}

\begin{remark}
In the special case \(\alpha = -3\), the operators \(\mathcal{L}_t\) and \(\mathcal{Q}_t\) reduce to their autonomous counterparts \(\mathcal{L}\) and \(\mathcal{Q}\), and the corresponding $L^{p}$ and $H^{s}$ estimates provided in Lemma \ref{Lemma:Lp} and Lemma \ref{Lemma:Sobolev} follow directly from the general (time-dependent) result.
\end{remark}

\section{Local-in-time existence and blow-up criteria for the $f$-model}\label{sec:3}
The main objective of this section is twofold. First, we establish the local-in-time existence and uniqueness of solutions to the $f$-model equation~\eqref{model2}, under suitable initial data in Sobolev spaces, see Theorem \ref{result:th}. More precisely, mentioned in Remark \ref{eleccion:R} we will only focus on the case wher $R(t)=e^{\alpha t}$ with $\alpha\in \mathbb{R}$. Second, we derive a blow-up criterion, see Theorem \ref{cont:criterion}, that characterizes the possible breakdown of solutions in finite time, thereby providing insight into the mechanisms that may lead to loss of regularity or singularity formation.

\begin{theorem}\label{result:th}
Let $R(t)=e^{\alpha t}$ with $\alpha\in\mathbb{R}$. Let $f_{0}(x)\in H^{2}(\T)$ with mean zero. Then, there exists $T>0$ and a unique zero mean local solution 
\[ f\in C^{1}\left([-T,T]); H^{2}(\T)\right), \]
of \eqref{model2} with $f(x,0)=f_{0}(x)$. 
\end{theorem}

\begin{proof}[Proof of Theorem \ref{result:th}]
In order to show the result, we will apply the Picard Theorem on Banach spaces \cite[Theorem 3.1, \S 3.2.2]{majda2001}. First, we show that the mean-zero property is preserved under the evolution. By integrating over the spatial domain and using the periodic boundary conditions we find that
\[ \int_{\T} f_{t} \ dx  = \int_{\T} \mathcal{L}_{t} f \ dx.\]
Since 
\[ \int_{\T}\mathcal{L}_{t}f (x) \dx = 2\pi \widehat{\mathcal{L}_{t}f}(0),\]
and recalling the Fourier multiplier of the operator $\mathcal{L}_{t}$ given in \eqref{def:L} we have that $\widehat{\mathcal{L}_{t}f}(0)=0$.
Thus
\[ \int_{\T} f(x,t)  \ dx = \int_{\T} f_{0}(x) \ dx.\]

To that purpose, we write \eqref{model2} together with the initial condition $f(x,0)=f_{0}(x)$ as
\begin{equation}\label{initialPicard}
	f_{t}=\mathscr{F}[f], \quad f(x,0)=f_{0}(x),
\end{equation}
where
\begin{equation*}
	\mathscr{F}[f]=\mathcal{L}_tf+\mathcal{Q}_t\frac{3K}{2R_0}e^{(3+\alpha)t}(f_x\mathcal{L}_tf)_x+3\mathcal{Q}_t\frac{K}{R_0}e^{(3+\alpha)t}(\mathcal{L}_tf_xf)_x
	+3\mathcal{Q}_t(\frac{5K}{2R_0}-\frac{\alpha K}{R_0})e^{(3+\alpha)t}(f_{x}f)_x.
\end{equation*} 
In order to prove that \eqref{initialPicard} has a solution, we will show that its integrated version has a solution and consider the operator 
\[ \Phi[f]=f_{0}+\int_{0}^{t}\mathscr{F}[f(\tau)] \ d\tau, \]
defined on $C\left([-T,T]); H^{2}_{0}(\T)\right)$ where $T>0$ is to be determined later on. We have denoted by $H^{2}_{0}(\T)$ the function space 
\[
H^2_0(\mathbb{T}) := \left\{ f \in H^2(\mathbb{T}) \; : \; \int_{\mathbb{T}} f(x)\, dx = 0 \right\}.
\]

 We will show that $\Phi$ is a bounded map on this space and moreover a contraction. Hence, Banach fixed point theorem will then yield a unique fixed point $f$ for the mapping $f\to \Phi[f]$. \medskip

First, let us prove that 
\begin{equation}\label{Fball} \mathscr{F}\left(B_{M}(H^{2}_{0})\right) \to B_{C_{0}(1+\lambda(t))M^{2}}(H^2_{0}), 
\end{equation}
for some positive constant $C_{0}>0$ independent of the radious $M>0$ and $\lambda(t)=\frac{K}{R_{0}}e^{(3+\alpha)t}$. The notation 
$B_{R}(H^2_{0})$ denotes de ball of radious $R$ in the Hilbert space $H^{2}_{0}$. Using the fact that $H^{2}_{0}$ is a Banach algebra, and (i), (ii) of Lemma \ref{Lemma:Sobolev} one can easily show that
\begin{align*}
\norm{\mathscr{F}[f]}_{H^2}&\leq C\bigg( \norm{f}_{H^2}+ \lambda(t)\norm{\mathcal{Q}_{t}(f_x\mathcal{L}_tf)_x}_{H^{2}}+\lambda(t)\norm{\mathcal{Q}_{t}(f\mathcal{L}_tf_{x})_x}_{H^{2}}+\lambda(t)\norm{\mathcal{Q}_{t}(f_x f)_x}_{H^{2}}\bigg) \\
& \leq C \left(\norm{f}_{H^2}+\lambda(t)\norm{f}_{H^{2}}^{2}\right),
\end{align*}
which shows \eqref{Fball}. \medskip

In the sequel, let us show that $\mathscr{F}$ is locally Lipschitz continous in $H^{2}_{0}$. To do so, we establish that for $f,g\in H^{2}_{0}$ we have that
\begin{equation}\label{lipschitz:condition}
	\displaystyle\sup_{f,g\in B_{M}(H^2)_{0}}\frac{\norm{\mathscr{F}[f]-\mathscr{F}[g]}_{H^{2}}}{\norm{f-g}_{H^2}}\leq C_{0}(1+M \lambda(t)).
	\end{equation}
where $C_{0}>0$ is the same constant as before (does not depende on $M$).
Using the fact that $\mathcal{L}_{t}$ is linear and the remaining terms are nonlinear we have that
\[
	\norm{\mathscr{F}[f]-\mathscr{F}[g]}_{H^{2}}\leq C \norm{\mathcal{L}_{t}(f-g)}_{H^2}+N_{1}+N_{2}+N_{3},
\]
where
\begin{align*}
N_{1}=	\lambda(t)&\norm{\mathcal{Q}_{t}\partial_{x}\left(f_{x}\mathcal{L}_{t}f-g_{x}\mathcal{L}_{t}g \right)}_{H^2},  \quad N_{2}=\lambda(t)\norm{\mathcal{Q}_{t}\partial_{x}\left(f\mathcal{L}_{t}f_{x}-g\mathcal{L}_{t}g_{x} \right)}_{H^2}, \\
&\quad \quad \quad \ \ \ N_{3}=\lambda(t)\norm{\mathcal{Q}_{t}\partial_{x}\left(ff_{x}-gg_{x} \right)}_{H^2}.
\end{align*}
Clearly, using (ii) in Lemma \ref{Lemma:Sobolev} we find that
\[\norm{\mathcal{L}_{t}(f-g)}_{H^2}\leq C \norm{f-g}_{H^2}. \]
To deal with the nonlinear terms, have to rewrite the terms by adding and subtracting. More precisely, we have that
\begin{align*}
|N_{1}|&\leq C \lambda(t) \norm{(f_{x}-g_{x})\mathcal{L}_{t}f+g_{x}(\mathcal{L}_{t}f-\mathcal{L}_{t}g)}_{H^1} \\
&\leq C \lambda(t)\left( \norm{f_{x}-g_{x}}_{H^1}\norm{\mathcal{L}_{t}f}_{H^1}+\norm{g_{x}}_{H^1}\norm{\mathcal{L}_{t}(f-g)}_{H^1}\right) \\
&\leq C \lambda(t) \left(\norm{f}_{H^2}+\norm{g}_{H^2}\right)\norm{f-g}_{H^2}.
\end{align*}
Arguing similary, one can show that
\begin{align*}
	|N_2| &\leq C \lambda(t) \left\| \mathcal{L}_t f_x \cdot f - \mathcal{L}_t g_x \cdot g \right\|_{H^1} \\
	&\leq C \lambda(t) \left( \| \mathcal{L}_t f_x - \mathcal{L}_t g_x \|_{H^1} \| f \|_{H^1} + \| \mathcal{L}_t g_x \|_{H^1} \| f - g \|_{H^1} \right) \\
	&\leq C \lambda(t) \left( \| f \|_{H^2} + \| g \|_{H^2} \right) \| f - g \|_{H^2},
\end{align*}
and
\begin{align*}
	|N_3| &\leq C \lambda(t) \left\| f_x f - g_x g \right\|_{H^1} \\
	&\leq C \lambda(t) \left( \| f_x - g_x \|_{H^1} \| f \|_{H^2} + \| g_x \|_{H^1} \| f - g \|_{H^1} \right) \\
	&\leq C \lambda(t) \left( \| f \|_{H^2} + \| g \|_{H^2} \right) \| f - g \|_{H^2}.
\end{align*}
Thus, we find that
\begin{equation}	\norm{\mathscr{F}[f]-\mathscr{F}[g]}_{H^{2}}\leq C\| f - g \|_{H^2}+ C\lambda(t) \left( \| f \|_{H^2} + \| g \|_{H^2} \right) \| f - g \|_{H^2},
\end{equation}
which shows \eqref{lipschitz:condition}. The proof now concludes as follows. Let $f_{0}\in H^{2}_{0}$ and define 
\[ M=2\norm{f_0}_{H^2}, \quad T=\displaystyle\min\{\frac{1}{2C_{0}M(1+\lambda^{\star})}, \frac{1}{2C_{0}(1+M\lambda^{\star})}\} \]
where $C_{0}$ is the constant above and 
\[ \lambda^{\star}=\displaystyle\sup_{t\in[-T,T]}\lambda(s).\] 
Consider the following space
\[ \mathcal{B}=C\left([-T,T]; \overline{B}_{M}(H^2_{0})\right).\]
	We claim that $\Phi[f]:\mathcal{B}\mapsto\mathcal{B}$ is a contraction. Indeed, we we first have that
	\begin{align}
		\norm{\Phi[f]}_{\mathcal{B}}\leq \norm{f_0}_{H^{2}}+\displaystyle\sup_{t\in[-T,T]}\int_{0}^{t}\norm{\mathscr{F}[f(s)]}_{H^2} \ ds \leq \frac{M}{2}+T\left( C_{0}(1+\lambda^{\star})M^{2} \right)=M.
	\end{align}
Moreover, the fact that it is a contraction follows from \eqref{lipschitz:condition}, namely
	\begin{align*}
	\norm{\Phi[f]-\Phi[g]}_{\mathcal{B}}&\leq \displaystyle\sup_{t\in[-T,T]}\int_{0}^{t}\norm{\mathscr{F}[f(s)]-\mathscr{F}[g(s)]}_{H^2} \ ds \\
	&\leq T C_{0}(1+M\lambda^{\star})) \| f - g \|_{H^2}=\frac{1}{2} \| f - g \|_{H^2}.
\end{align*}
Invoking Banach fixed point argument we conclude that there exists a unique fixed point $f\in \mathcal{B}$ of the mapping $f\to \Phi[f]$. Taking a time derivative, we have thus obtained a local in time solution to \eqref{model2} in $H^{2}$.
\end{proof}

\begin{remark}
	The result of Theorem~\ref{result:th} can be extended, with minor modifications, to show the existence and uniqueness of solutions to \eqref{model2} in the H\"older space
	\[
	f \in C^1\left([-T,T], C^{1,\alpha}(\mathbb{T})\right), \quad \text{for some } 0 < \alpha < 1,
	\]
	for a certain time \( T > 0 \). The key change in the argument is that, instead of applying Lemma~\ref{Lemma:Sobolev} to estimate the operators \( \mathcal{Q}_t \) and \( \mathcal{L}_t \) in the \( H^2 \) setting, one can use analogous Schauder estimates in the H\"older spaces \( C^{k,\alpha} \), cf. \cite[Theorem 4, \S4.4 ]{SteinDiff} or \cite[Theorem 6.19, \S6.6]{abels2012}. To avoid repetition, we omit the details of this adaptation.
\end{remark}

\begin{remark}
	We observe that the existence and uniqueness result of Theorem~\ref{result:th} applies in particular to the autonomous case \eqref{model2:aut}. Indeed, when choosing \( \alpha = -3 \), the coefficient \( R(t) = e^{\alpha t} \) becomes \( R(t) = e^{-3t} \), and the time-dependent operators \( \mathcal{Q}_t =\mathcal{Q} \)and \( \mathcal{L}_t=\mathcal{L} \) become independent of time.  In this case, equation \eqref{model2} reduces exactly to the autonomous model \eqref{model2:aut}. Therefore, Theorem~\ref{result:th} ensures the local well-posedness of \eqref{model2:aut} in both the Sobolev setting \( H^2(\mathbb{T}) \) and the H\"older space \( C^{1,\alpha}(\mathbb{T}) \), as discussed in the previous remark.

\end{remark}

We now establish a continuation criterion which characterizes precisely when a local solution to equation \eqref{model2} can be extended beyond its maximal time of existence. 
\begin{theorem}[Continuation criterion]\label{cont:criterion}
	Let \( f_0 \in H^2(\mathbb{T}) \) be a function with zero mean, and let \( T > 0 \) denote the maximal time of existence for a solution 
	\[
	f \in C^1([0, T), H^2),
	\]
	to equation \eqref{model2} with initial condition \( f(x,0) = f_0(x) \). Then the solution \( f(x,t) \) can be continued beyond time \( T \) up to some \( T^\star > T \) if and only if
	\begin{equation}
		\int_0^{T^\star} \lambda(t) \left(1 + \|f(t)\|_{L^\infty} \right)\, dt < \infty.
	\end{equation}
\end{theorem}

\begin{proof}[Proof of Theorem \ref{cont:criterion}]	
Testing equation \eqref{model2} with $f$, integrating by parts and using that the operator $\mathcal{Q}$ is self-adjoint, we find that
	\[ 
	\frac{1}{2}\frac{d}{dt}\norm{f}_{L^{2}}^{2} =\int_{\T} \mathcal{L}_{t}f f \ dx+I_{1}+I_{2}+I_{3},
	\]
	where
	\begin{align*} 
		I_{1}=\frac{3}{2}\lambda(t)\int_{\T} f_{x}&\mathcal{L}_{t}f \mathcal{Q}_{t}f_{x} \ dx, \quad I_{2}= - 3\lambda(t) \int_{\T} f\mathcal{L}_{t}f_{x}\mathcal{Q}_{t} f_{x}\ dx, \\
		I_{3}&=-3\left(\frac{5}{2}+\alpha\right)\lambda(t)\int_{\T} f_{x}f \mathcal{Q}_{t}f_{x} \ dx.
	\end{align*}
	For notational convenience, we set $\lambda(t) := \frac{K}{R_0} e^{(3 + \alpha)t}.$ Using the identity (i) of Lemma \ref{Lemma:Lp} we find that
	\[ I_{1}=- \frac{3}{2}\left(\frac{5}{2}+\alpha\right) \lambda(t)\int_{\T}f_{x}f \mathcal{Q}_{t}f_{x} \ dx+ \frac{3}{2}\left(\frac{5}{2}+\alpha\right) \lambda(t)\int_{\T}f_{x}\mathcal{Q}_{t}f \mathcal{Q}_{t}f_{x} \ dx =I_{11}+I_{12}. \]
	
	Integrating by parts in $I_{11}$, noticing that 
	\begin{equation}\label{relation:QL} \mathcal{L}_{t}f=(\frac{5}{2}+\alpha)\lambda(t)\mathcal{Q}_{t}f_{xx}
	\end{equation}
	and using (iii) in Lemma \ref{Lemma:Lp} we find that
	\begin{equation}\label{est:I11}
		|I_{11}|= \abs{\frac{3}{4}\left(\frac{5}{2}+\alpha\right) \lambda(t)\int_{\T} f^2\mathcal{Q}_{t}f_{xx} \ dx} \leq C \norm{\mathcal{L}_{t}f}_{L^\infty} \norm{f}_{L^2}^{2} \leq C  \norm{f}_{L^\infty} \norm{f}_{L^2}^{2}.
	\end{equation}
	
	Similarly, integrating by parts in $I_{12}$ and using relation \eqref{relation:QL}
	\begin{equation}\label{term:I12}
		I_{12}=- \frac{3}{2}\left(\frac{5}{2}+\alpha\right) \lambda(t)\int_{\T} f (\mathcal{Q}_{t}f_{x})^{2} \ dx -  \frac{3}{2}\int_{\T} f \mathcal{Q}_{t}f \mathcal{L}_{t} f\ dx 
	\end{equation}
	Thus, by means of Lemma \ref{Lemma:Lp} for $p=2$ we conclude that
	\begin{equation}\label{est:I12}
		\abs{I_{12}}\leq C \lambda(t)  \norm{f}_{L^{\infty}}\norm{f}_{L^2}^{2}.
	\end{equation}
	Combining \eqref{est:I11} and \eqref{est:I12} we find that
	\begin{equation*}
		\abs{I_{1}}\leq C \lambda(t)  \norm{f}_{L^{\infty}}\norm{f}_{L^2}^{2}.
	\end{equation*}
	In order to bound $I_{2}$ we make use of identity (i) of Lemma \ref{Lemma:Lp} to write
	\begin{equation*}
		I_{2}=3(\frac{5}{2}+\alpha)\lambda(t) \int_{\T} ff_{x} \mathcal{Q}_{t}f_{x} \ dx -3(\frac{5}{2}+\alpha)\lambda(t)\int_{\T} f (\mathcal{Q}_{t}f_{x})^{2} \ dx=I_{21}+I_{22}
	\end{equation*}
	Noticing that $I_{21}=-I_{3}$ and that $I_{22}$ equals twice the first integral on the right-hand side in \eqref{term:I12} we conclude that
	\begin{equation}\label{I1:I3}
		\abs{I_{1}+I_{2}+I_{3}}\leq C  \lambda(t)  \norm{f}_{L^{\infty}}\norm{f}_{L^2}^{2}.
	\end{equation}
	Moreover, invoking H\"older's inequality and (ii) of Lemma  \ref{Lemma:Lp} with $p=2$ we can bound the linear term as
	\begin{equation}\label{linear:termL2}
		\abs{\int_{\T} \mathcal{L}_{t}f f \ dx} \leq C\norm{f}_{L^{2}}^{2}. 
	\end{equation}
	Thus estimates \eqref{I1:I3} and \eqref{linear:termL2} yields
	\begin{equation}\label{estL2}
		\frac{1}{2}\frac{d}{dt}\norm{f}_{L^{2}}^{2}\leq C  \lambda(t) \left(1+ \norm{f}_{L^{\infty}}\right)\norm{f}_{L^2}^{2}.
	\end{equation}

	Next we compute the evoution of the $\dot{H^2}$ norm. To that purpose, testing \eqref{model2} against $f_{xxxx}$ and integrating by parts
	\[ 
	\frac{1}{2}\frac{d}{dt}\norm{f_{xx}}_{L^{2}}^{2}-\int_{\T} \mathcal{L}_{t}f_{xx}f_{xx} \ dx =J_{1}+J_{2}+J_{3}
	\]
	where
	\begin{align*} J_{1}=\frac{3}{2}\lambda(t)\int_{\T}(f_x&\mathcal{L}_tf)_x\mathcal{Q}_tf_{xxxx} \ dx, \quad J_{2}=3\lambda(t)\int_{\T} (\mathcal{L}_tf_xf)_x\mathcal{Q}_tf_{xxxx} \ dx, \\
		J_{3}&=	3(\frac{5}{2}+\alpha)\lambda(t)\int_{\T}(f_{x}f)_x \mathcal{Q}_{t}f_{xxxx} \ dx
	\end{align*}
	By means of relation \eqref{relation:QL}, we can write
	\begin{equation}
		J_{1}=\frac{3}{2(\frac{5}{2}+\alpha)}\ \int_{\T} \left( f_{xx}\mathcal{L}_{t}f + f_{x} \mathcal{L}_{t}f_{x}\right) \mathcal{L}_{t} f_{xx} \ dx.
	\end{equation}
	Hence, using H\"older's inequality, (ii) in Lemma \ref{ine:Sobolev} and (ii) in Lemma \ref{Lemma:Lp} we have that
	\begin{align}
		\abs{J_{1}} &\leq C\left( \norm{\mathcal{L}_{t}f}_{L^{\infty}}\norm{f_{xx}}_{L^2}\norm{\mathcal{L}_{t}f_{xx}}_{L^{2}}+\norm{f_{x}}_{L^4}\norm{\mathcal{L}_{t}f_{x}}_{L^{4}}\norm{\mathcal{L}_{t}f_{xx}}_{L^{2}}\right) \nonumber \\
		&\leq C \norm{f}_{L^{\infty}}\norm{f_{xx}}_{L^{2}}^{2}. \label{J1}
	\end{align}
	To estimate $J_{2}$ aand $J_{3}$ we follow the same strategy, we first write using \eqref{relation:QL}
	\[J_{2}= \frac{3}{(\frac{5}{2}+\alpha)}\ \int_{\T} \left( \mathcal{L}_{t}f_{xx}f + f_{x} \mathcal{L}_{t}f_{x}\right) \mathcal{L}_{t} f_{xx} \ dx, \]
	and
	\[J_{3}= 3\int_{\T} f_{x}^{2} \mathcal{L}_{t}f_{xx} + f f_{xx} \mathcal{L}_{t}f_{xx} \ dx.  \]
	Thus, as before, we readily check that
	\begin{align}
		\abs{J_{2}} &\leq C\left( \norm{f}_{L^{\infty}}\norm{\mathcal{L}_{t}f_{xx}}_{L^2}^{2}+\norm{f_{x}}_{L^4}\norm{\mathcal{L}_{t}f_{x}}_{L^{4}}\norm{\mathcal{L}_{t}f_{xx}}_{L^{2}}\right) \nonumber  \\
		&\leq C \norm{f}_{L^{\infty}}\norm{f_{xx}}_{L^{2}}^{2}. \label{J2}
	\end{align}
	and
	\begin{align}
		\abs{J_{3}} &\leq C\left( \norm{f_{x}}_{L^{4}}\norm{\mathcal{L}_{t}f_{xx}}_{L^2}+\norm{f}_{L^\infty}\norm{f_{xx}}_{L^{2}}\norm{\mathcal{L}_{t}f_{xx}}_{L^{2}}\right) \nonumber \\
		&\leq C \norm{f}_{L^{\infty}}\norm{f_{xx}}_{L^{2}}^{2}. \label{J3}
	\end{align}
	Collecting estimates \eqref{J1}-\eqref{J3} we have shown that
	\[ 
	\frac{1}{2}\frac{d}{dt}\norm{f_{xx}}_{L^{2}}^{2}-\int_{\T} \mathcal{L}_{t}f_{xx}f_{xx} \ dx \leq C \norm{f}_{L^{\infty}}\norm{f_{xx}}_{L^{2}}^{2}.
	\]
	As before,  (ii) of Lemma \ref{Lemma:Lp} shows that
	\[ \int_{\T} \mathcal{L}_{t}f_{xx} f_{xx} \ dx \leq C \norm{f_{xx}}^{2}_{L^2},\]
	and thus
	\begin{equation}\label{est:H2}
		\frac{1}{2}\frac{d}{dt}\norm{f_{xx}}_{L^{2}}^{2} \leq C \norm{f}_{L^{\infty}}\norm{f_{xx}}_{L^{2}}^{2}.
	\end{equation}
	Combining estimates \eqref{estL2} and \eqref{est:H2}, we find that
	\[
	\frac{1}{2}\frac{d}{dt} \|f(t)\|_{H^2}^2 \leq C \lambda(t) \left(1 + \|f(t)\|_{L^\infty} \right) \|f(t)\|_{H^2}^2.
	\]
	An application of Grönwall’s inequality yields
	\[
	\sup_{t \in [0,T]} \|f(t)\|_{H^2} \leq \|f_0\|_{H^2} \exp\left( \int_0^T \lambda(t) \left(1 + \|f(t)\|_{L^\infty} \right) dt \right),
	\]
	as long as the integral on the right-hand side remains finite.
	
	Now, let \( T^\star > T \) and suppose
	\[
	\int_0^{T^\star} \lambda(t) \left(1 + \|f(t)\|_{L^\infty} \right) dt < \infty.
	\]
	Then, by applying the previous estimate on any compact subinterval $[0, T_1] \subset [0, T^\star) $, we obtain a uniform bound on $ \|f(t)\|_{H^2} $. In particular, this implies that the \( H^2 \)-norm does not blow up as \( t \to T^\star \), and hence, by a standard continuation argument, the solution can be extended up to time \( T^\star \).  On the other hand, assume that the solution cannot be extended beyond time \( T^\star \). Then necessarily \( \|f(t)\|_{H^2} \to \infty \) as \( t \to T^\star \). Since the embedding \( H^2(\mathbb{T}) \hookrightarrow L^\infty(\mathbb{T}) \) is continuous, it follows that either \( \|f(t)\|_{L^\infty} \to \infty \), or it grows too fast to keep the integral finite. In either case,
	\[
	\int_0^{T^\star} \lambda(t) \left(1 + \|f(t)\|_{L^\infty} \right) dt = \infty.
	\]

\end{proof}

\begin{remark}
	When \( R(t) = e^{-3t} \), the equation \eqref{model2} reduces to the autonomous model \eqref{model2:aut}.	In this case, the continuation criterion stated in Theorem~\ref{cont:criterion} simplifies considerably, since the weight \( \lambda(t) = e^{(3+\alpha)t} \) becomes constant, i.e., \( \lambda(t) \equiv 1 \). Consequently, the solution to \eqref{model2:aut} can be continued beyond time \( T \) if and only if
	\[
	\int_0^{T^\star} \left(1 + \|f(t)\|_{L^\infty} \right)\, dt < \infty.
	\]
In the following section, we will apply this criterion in a numerical setting to investigate whether solutions to \eqref{model2:aut} exhibit finite-time blow-up.
\end{remark}

\section{Global in time result for the $f$-model equation}\label{sec:4}
In this section, we provide a global in time result for solutions to the $f$-model equation \eqref{model2} with $R(t)=e^{\alpha t}$ for specific values of the parameter $\alpha$. The result reads as follows:

\begin{theorem}\label{theorem:global}
Let $R(t)=e^{\alpha t}$ with $\alpha>-2 $. Let $f_{0}(x)\in H^{2}(\T)$ with mean zero. Then, if 
$\norm{f_0}_{H^2(\T)}$ is sufficiently small, 
 there exists a global-in-time unique global solution 
\[ f\in C\left([0,\infty); H^{2}(\T)\right), \]
of \eqref{model2} with $f(x,0)=f_{0}(x)$. 
\end{theorem}

\begin{remark}
The restriction on the parameter \( \alpha \), such as \( \alpha > -2 \), ensures that the exponential weight in the energy functional contributes a positive term. This can be interpreted as an effective dissipation quantity that helps control the evolution of solutions over time. The proof of Theorem~\ref{theorem:global} makes explicit the smallness condition required on the initial data. Specifically, it shows that the norm $\norm{f_0}_{H^2(\T)}<\delta$ where $\delta$ is a  positive constant that depends only on fixed parameters of the problem.
\end{remark}

\begin{proof}[Proof of Theorem \ref{theorem:global}]
The existence of a local-in-time unique solution up to a certain maximal time of existence $T>0$ is by Theorem \ref{result:th}. In the following we will show that we can derive a global-in-time control of the $\dot{H}^{2}$ norm of the solution for $\alpha>-2 $. First, notice that for $R(t)=e^{\alpha t}$, equation \eqref{model2} is equivalent to
\begin{multline}\label{modified:eq}
f_t-\frac{K}{R_0}e^{(3+\alpha)t}f_{xxt}-(\frac{5K}{2R_0}+\frac{K\alpha}{R_0})e^{(3+\alpha)t}f_{xx}\\=\frac{3K}{2R_0}e^{(3+\alpha)t}(f_x\mathcal{L}_tf)_x+3\frac{K}{R_0}e^{(3+\alpha)t}(\mathcal{L}_tf_xf)_x
+3(\frac{5K}{2R_0}+\frac{K\alpha}{R_0})e^{(3+\alpha)t}(f_{x}f)_x.
\end{multline}
Testing agains $-f_{xx}$ and integrating by parts, we find that
\begin{equation}\label{eq:estimatcion}
\frac{1}{2}\frac{d}{dt}\norm{f_{x}}_{L^2}^{2} + \frac{K}{R_0}e^{(3+\alpha)t}\left( \int_{\T} f_{xxt} f_{xx} \ dx \right) + \frac{K}{R_0}e^{(3+\alpha)t}\left(\frac{5}{2}+\alpha\right)\norm{f_{xx}}_{L^2}^{2}= K_{1}+K_{2}+K_{3},
\end{equation}
with 
\[ K_{1}=- \frac{3K}{2R_0}e^{(3+\alpha)t} \int_{\T} (f_x\mathcal{L}_tf)_x f_{xx} \ dx, \quad \quad  K_{2}=-3\frac{K}{R_0}e^{(3+\alpha)t}\int_{\T}(\mathcal{L}_tf_xf)_x f_{xx} \ dx \]
and
\[ K_{3}= -3\frac{K}{R_0}\left(\frac{5}{2}+\alpha\right)e^{(3+\alpha)t}\int_{\T}(f_{x}f)_x f_{xx} \ dx. \]
Next, we notice that the second term on the left hand side in \eqref{eq:estimatcion} can be written as
\begin{equation}
\frac{K}{R_0}e^{(3+\alpha)t}\left( \int_{\T} f_{xxt} f_{xx} \ dx \right)=\frac{d}{dt} \left( \frac{K}{2R_0}e^{(3+\alpha)t} \norm{f_{xx}}_{L^{2}}^{2}\right)-\frac{K(3+\alpha)}{2R_0}e^{(3+\alpha)t} \norm{f_{xx}}_{L^2}^{2}.
\end{equation}
Thus, \eqref{eq:estimatcion} takes the form
\begin{equation*}
\frac{1}{2}\frac{d}{dt}\bigg[ \norm{f_{x}}_{L^2}^{2}+ \left( \frac{K}{R_0}e^{(3+\alpha)t} \norm{f_{xx}}_{L^{2}}^{2}\right) \bigg]+\gamma e^{(3+\alpha)t} \norm{f_{xx}}_{L^2}^{2}= K_{1}+K_{2}+K_{3},
\end{equation*}
where $\gamma=\frac{K}{R_0}\left(1+\frac{\alpha}{2}\right)$. Notice that $\gamma>0$ as long as $\alpha>-2$. In order to bound the nonlinear terms $K_{1},K_{2}$ and $K_{3}$ we estimate them similarly as we bounded $J_{1},J_{2}$ and $J_{3}$ in the proof of Theorem \ref{cont:criterion}. Indeed, by using (i) in Lemma \ref{ine:Sobolev} and (iii) in Lemma \ref{Lemma:Lp} we find that
\[ \abs{K_1+K_{2}+K_{3}}\leq C \frac{3K}{2R_0}e^{(3+\alpha)t} \left( \norm{f}_{L^{\infty}}\norm{f_{xx}}_{L^2}^{2}\right).\]
Therefore, denoting by
\[ \mathcal{E}(t)= \norm{f_{x}}_{L^2}^{2}+ \frac{K}{R_0}e^{(3+\alpha)t} \norm{f_{xx}}_{L^{2}}^{2}, \quad \mathcal{D}(t)=e^{(3+\alpha)t} \norm{f_{xx}}_{L^2}^{2}\]
and using the Sobolev embedding (ii) in Lemma \ref{ine:Sobolev}, we infer that
\begin{equation}\label{des:energia}
\frac{d}{dt}\mathcal{E}(t)+\gamma \mathcal{D}(t) \leq C\mathcal{D}(t) \sqrt{\mathcal{E}(t)}.
\end{equation}

Our goal is to show that if the initial energy \( \mathcal{E}(0) \) is sufficiently small, then \( \mathcal{E}(t) \) remains small for all time, and in fact decays.  Let us define the threshold $\delta = \frac{\gamma}{2C}$ and assume that
$ \mathcal{E}(0) < \delta^2$.
We define
\[
T^* = \sup\left\{ T > 0 : \mathcal{E}(t) < \delta^2 \text{ for all } t \in [0, T) \right\}.
\]
Suppose, by contradiction, that $T^\star, < \infty $. Then, by the definition of $T^\star$, we have
\[
\mathcal{E}(t) < \delta^2 \quad \text{for all } t \in [0, T^\star,).
\]
Inserting this into inequality~\eqref{des:energia}, we obtain
\[
\frac{d}{dt} \mathcal{E}(t) \leq \left( C \sqrt{\mathcal{E}(t)} - \gamma \right) \mathcal{D}(t) \leq \left( C \delta - \gamma \right) \mathcal{D}(t) = -\frac{\gamma}{2} \mathcal{D}(t) \leq 0.
\]
Thus, \( \mathcal{E}(t) \) is non-increasing on \( [0, T^*) \), and in particular,
\[
\mathcal{E}(t) \leq \mathcal{E}(0) < \delta^2 \quad \text{for all } t \in [0, T^*).
\]

Since \( \mathcal{E}(t) \) is continuous, we may take the limit as \( t \nearrow T^* \) and obtain
\[
\mathcal{E}(T^*) = \lim_{t \nearrow T^*} \mathcal{E}(t) \leq \mathcal{E}(0) < \delta^2.
\]
Therefore, there exists \( \varepsilon > 0 \) such that
\[
\mathcal{E}(T^*) < \delta^2 - \varepsilon.
\]
By continuity of \( \mathcal{E}(t) \), there exists \( \eta > 0 \) such that
\[
\mathcal{E}(t) < \delta^2 \quad \text{for all } t \in [T^*, T^* + \eta),
\]
contradicting the definition of \( T^* \) as the supremum of times for which \( \mathcal{E}(t) < \delta^2 \). Hence, we conclude that
\[
\mathcal{E}(t) < \delta^2 \quad \text{for all } t \geq 0.
\]

\end{proof}


\section{Numerical simulations}\label{sec:5}

%

In this section, we discuss some numerical simulations of the models under study. In
practical applications, physical values for $K$ in equations (\ref{biofilm_orig}), (\ref{biofilm}), 
(\ref{biofilm_orig_x}),  (\ref{biofilm_x}) are of order $K\sim 10^{-5}$. As commented
in the introduction, equation (\ref{biofilm_orig}) with $\delta=0$ is known to have approximate
selfsimilar solutions \cite{entropy} of the form 
$h =   {e^{t}  \over (R(t)/R_0)^2} \left( 1-{3 \over 2} {r^2 \over R(t)^2}\right)^{1/3},$
for specific choices of $R(t)$, see Figure \ref{selfsimilar} (a). 
Dividing by $e^t$, we obtain scaled profiles which
are approximate solutions of (\ref{biofilm}). 
These solutions are generated for the choice
$R(t)/R_0 = R_s(t) =\left(1 + {7 \over 3} \left(1+{3\over 2}\right) K (e^{3t} -1)\right)^{1/7}$, see \cite{entropy}.
Matching these profiles with tails of small height $h_{\rm inf}$, we can insert them as initial 
conditions in (\ref{biofilm}) and obtain solutions composed of a central part displaying 
self-similar behavior and two tails of height $h_{\rm inf}$.
A similar phenomenon is observed when the initial condition is a peak, see 
Figure \ref{bump2} (a) for $R(t)/R_0 = R_s(t) \sim e^{3t/7}$. Similar patterns arise for other
choices of $R(t) = e^{\alpha t}$. We illustrate the results when $\alpha =1$ and $\alpha=-1$,
satisfying the condition $\alpha > -2$.

\begin{figure}[!hbt]
\hskip 0.5cm (a) \hskip 7cm (b) \\
\includegraphics[width=7.5cm]{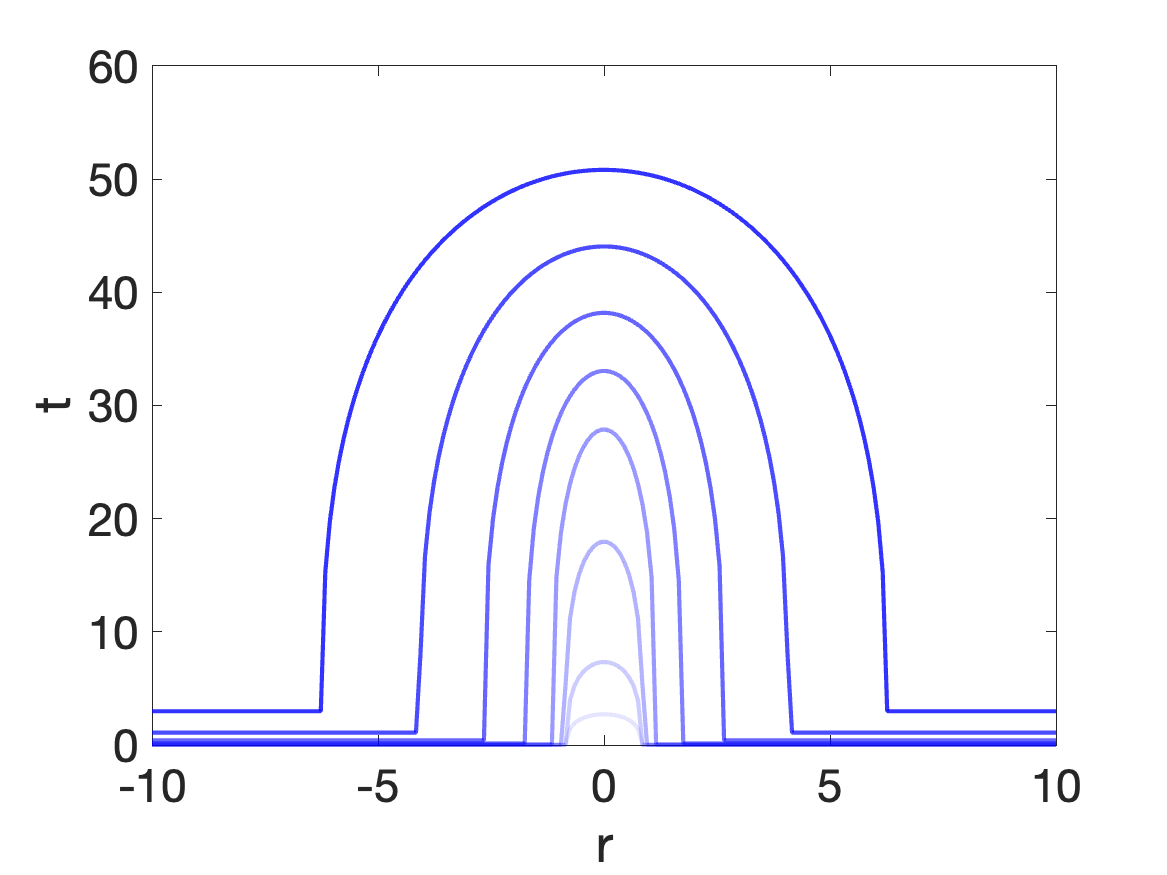}
\includegraphics[width=7.5cm]{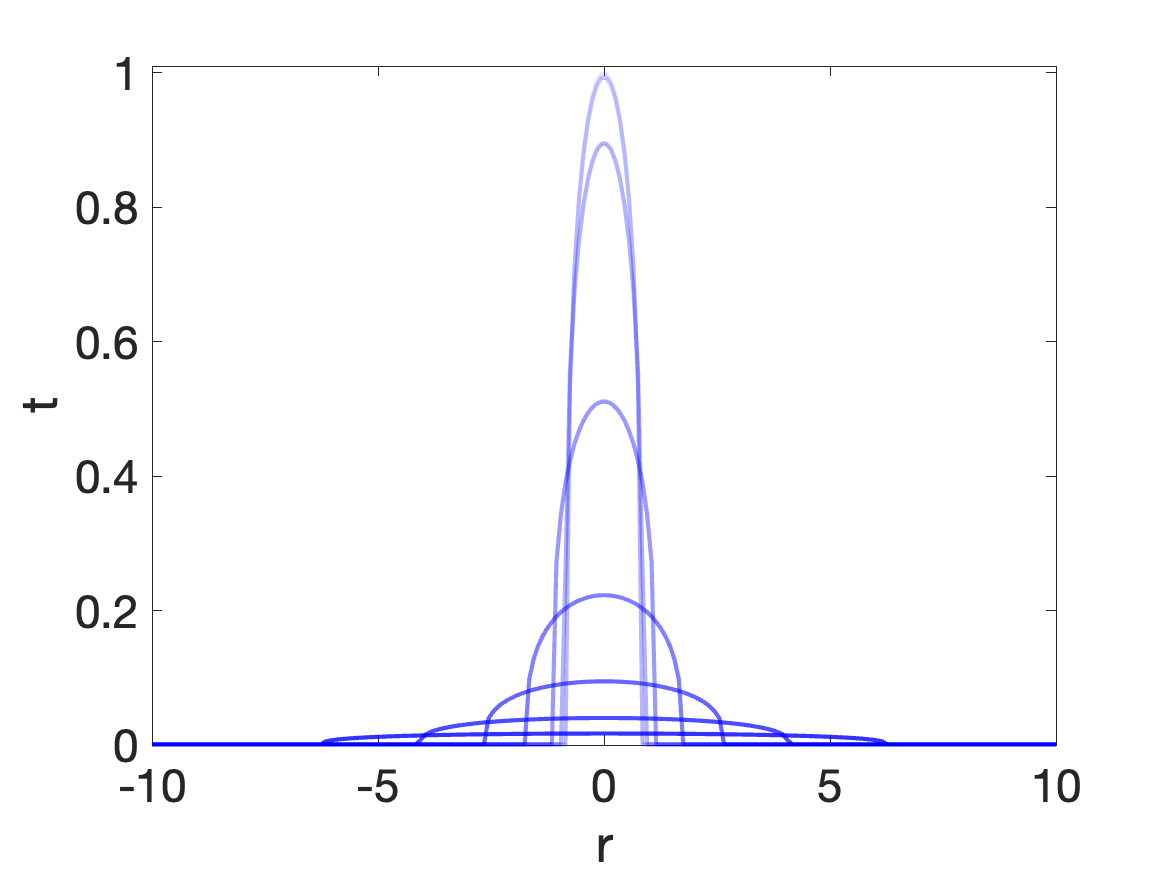}
\caption{(a) Selfsimilar solution of (\ref{biofilm_orig}) with 
$ R(t)/R_0 = R_s(t)$ at 
times $1,2,3,4,5,6,7,8$. The line color darkens slightly as $t$ grows.
(b) Scaled height profiles $\tilde h$ after division by $e^t$. The 
lighter profile corresponds to $t=1$.
Parameters: $\delta =0$, $K=10^{-5}$.}
\label{selfsimilar}
\end{figure}

In the evolution of these initial peaks, we observe a competition between the dynamics of the upper
part, that moves upwards, and the dynamics of the two lateral fronts, that advance sidewards.
Initially, the upward deformation dominates. In the two tests with $R(t) \sim e^{\alpha t}$, 
$\alpha>0$, we identify a threshold time after which sidewards motion becomes relevant and
slows down (even diminishes) the height variation. The central part flattens. As $\alpha$
decreases, the lateral fronts remain essentially pinned for longer times, while the height grows.
Studying the long time behavior of these patterns numerically is a challenge due to the need to
reduce the time step as the exponential factors present in the equation grow.
For the simulation considered in Figure \ref{bump2} (b) and (e), numerical artifacts are observed: small amplitude oscillations develop in the central area, which becomes flat but wavy.
This artifacts are delayed as we reduce the time step $dt$. 
For the simulation considered in Figure \ref{bump2} (c) and (f), the step gradients at both sides
can lead to numerical instability as time grows if we keep the steps fixed.
A similar phenomenon is observed working with (\ref{biofilm_orig_x}) and  (\ref{biofilm_x}) 
instead, see Figure \ref{bump3}.

\begin{figure}[!hbt]
\hskip 0.5cm (a) \hskip 4cm (b) \hskip 4cm (c)\\
\includegraphics[width=5cm]{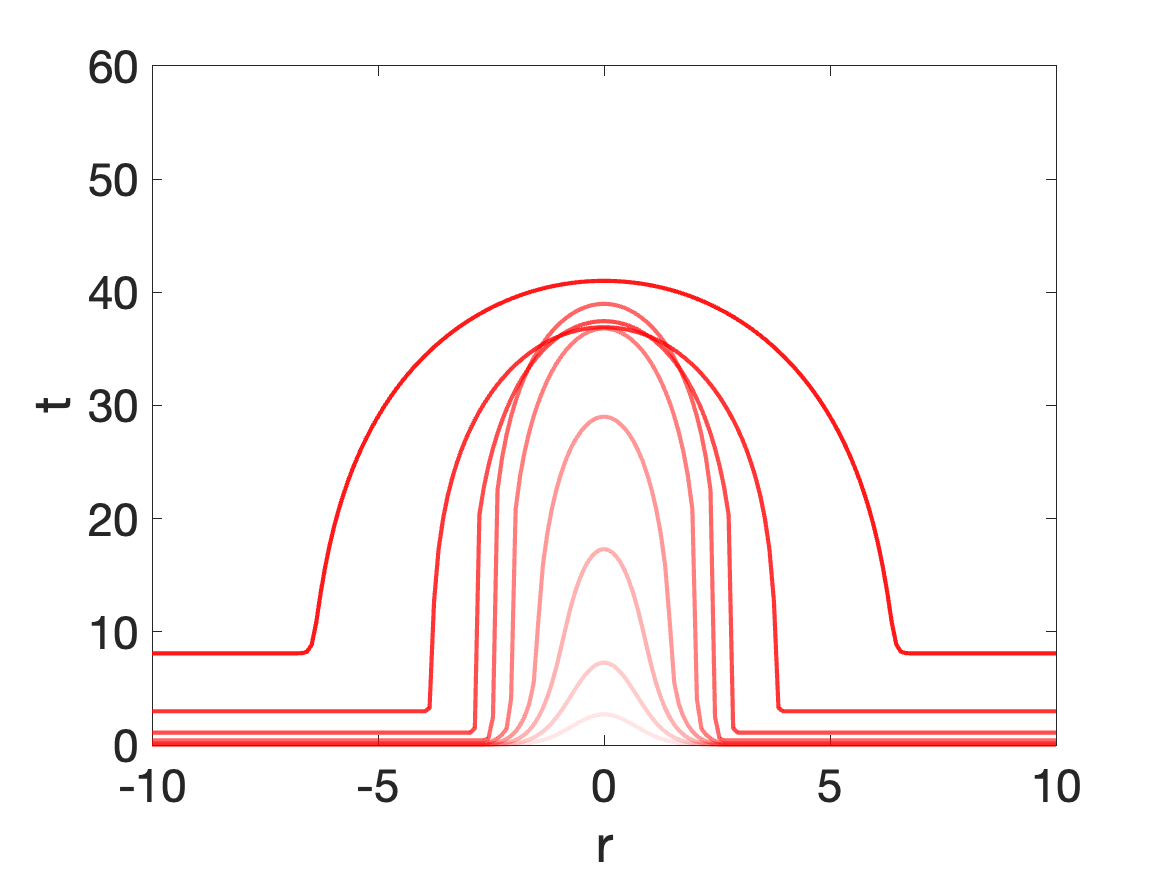}
\includegraphics[width=5cm]{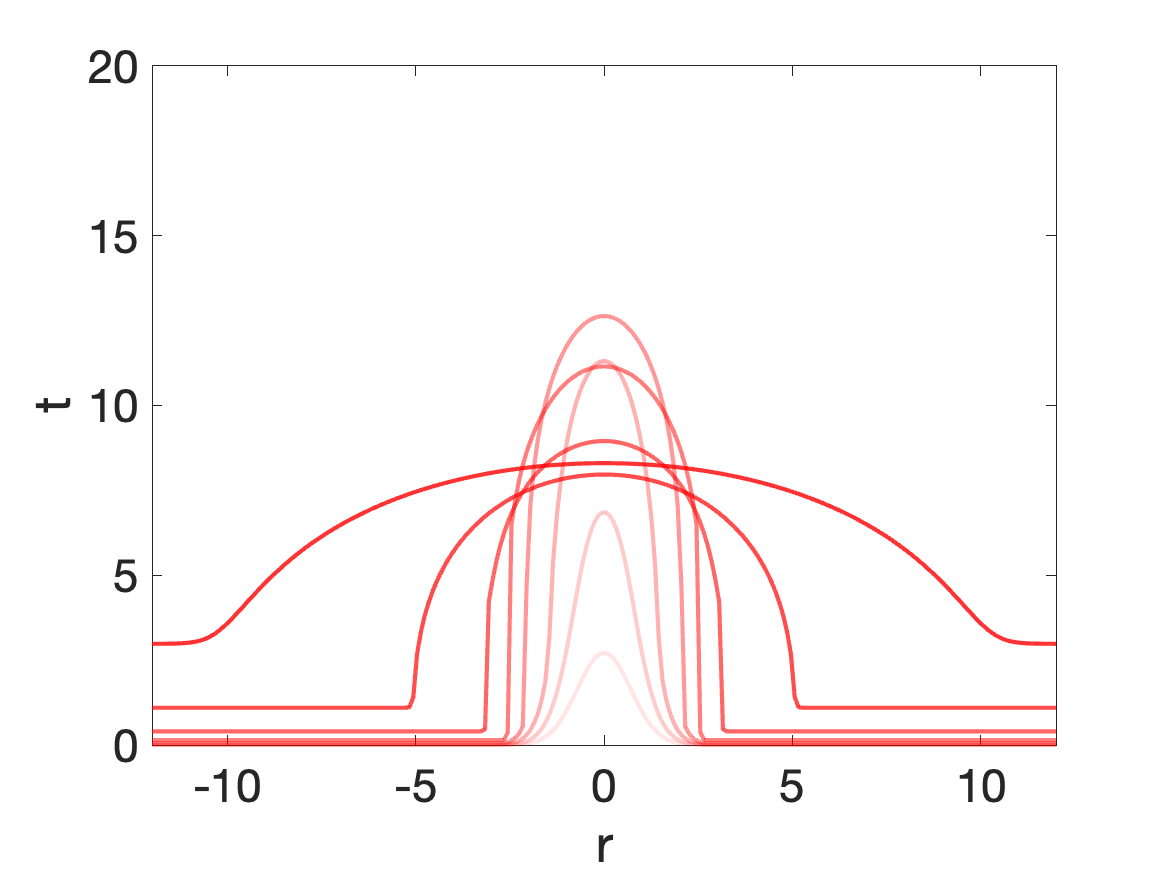}
\includegraphics[width=5cm]{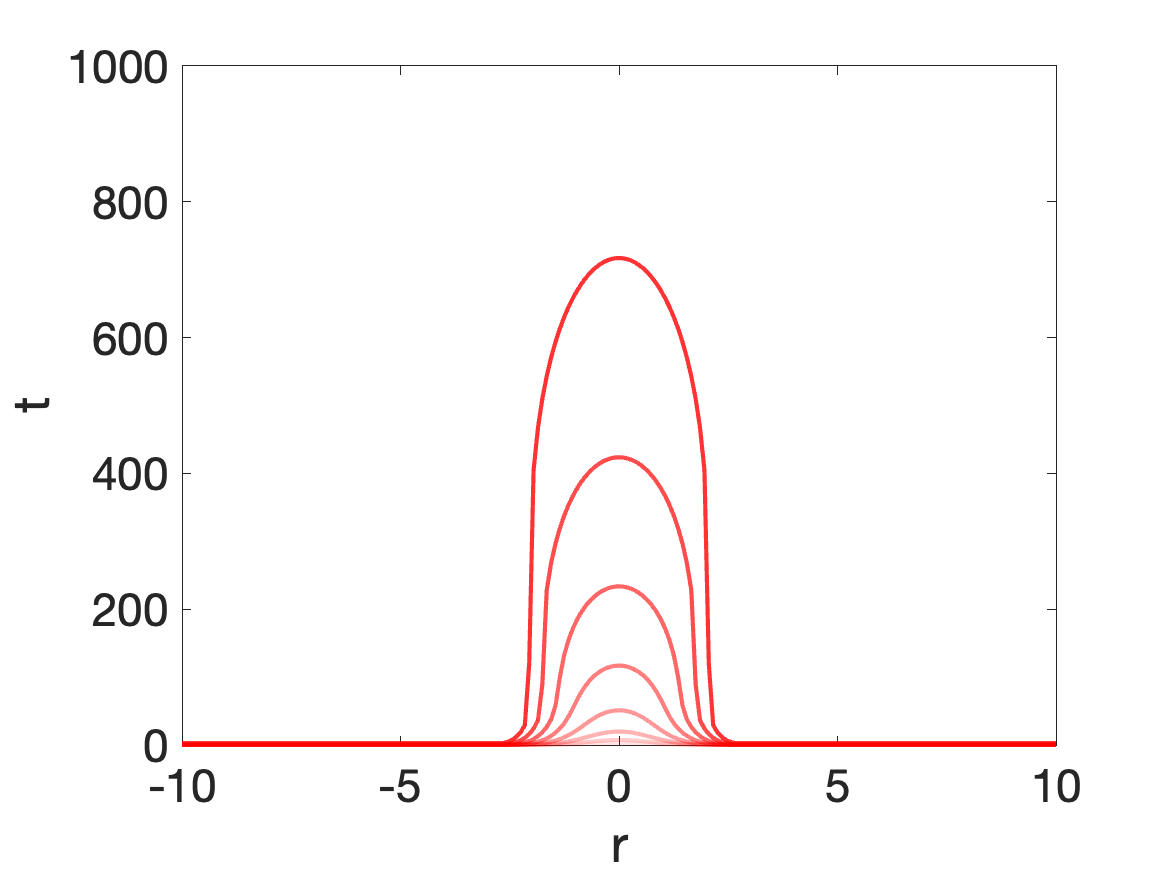} \\
\hskip 0.5cm (d) \hskip 4cm (e) \hskip 4cm (f)\\
\includegraphics[width=5cm]{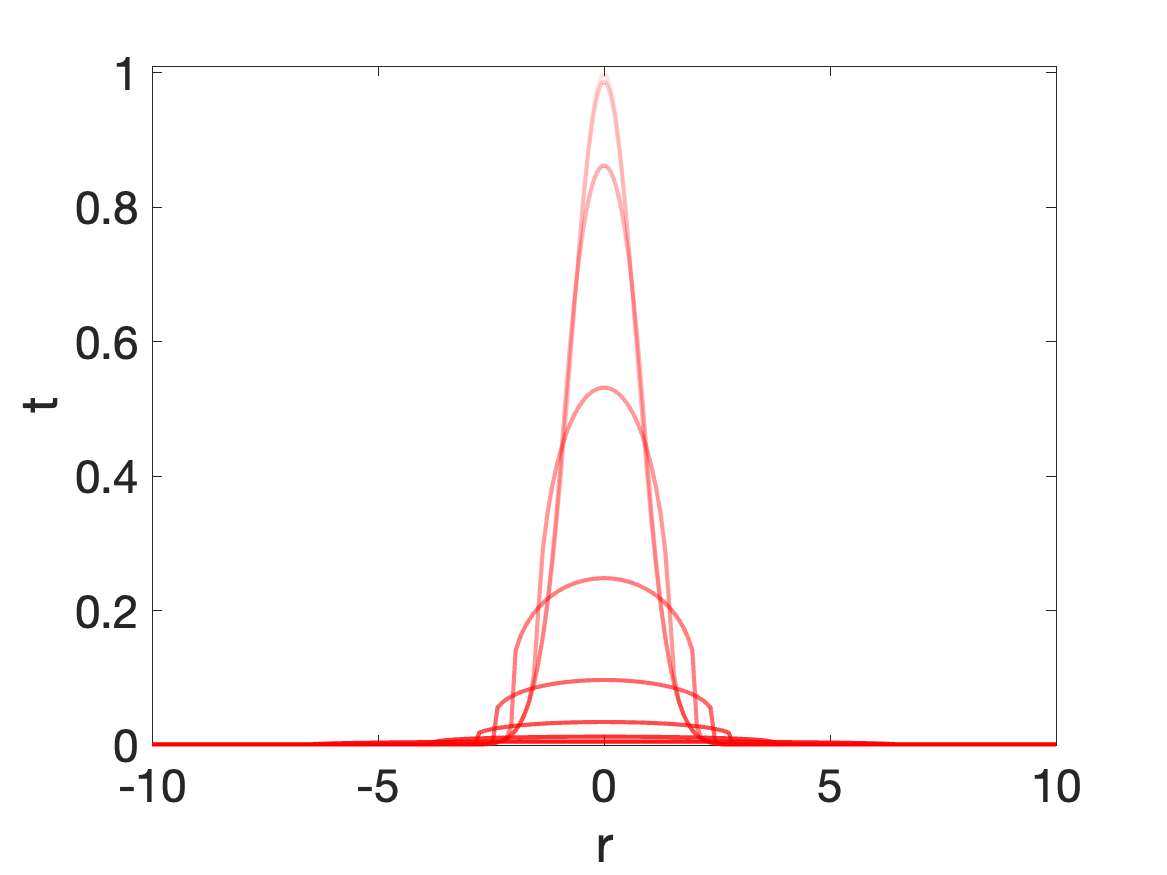}
\includegraphics[width=5cm]{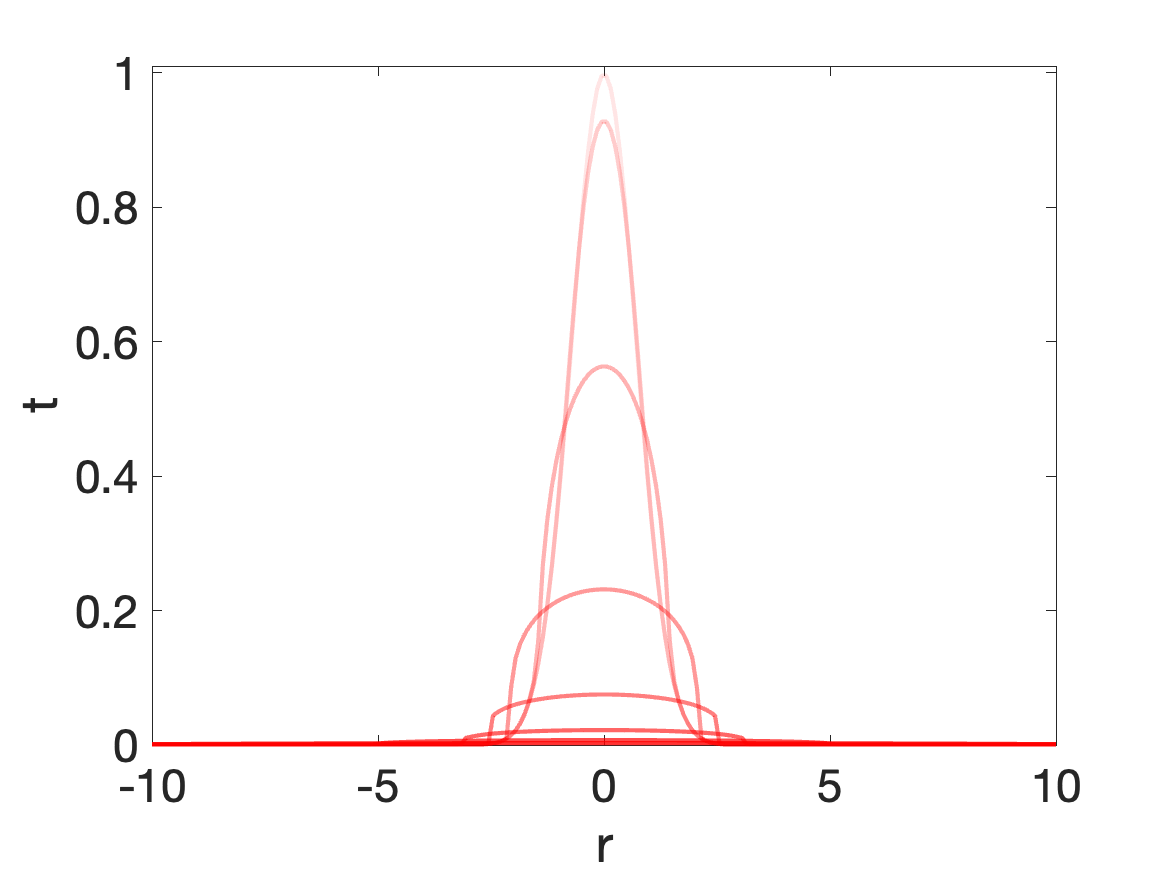}
\includegraphics[width=5cm]{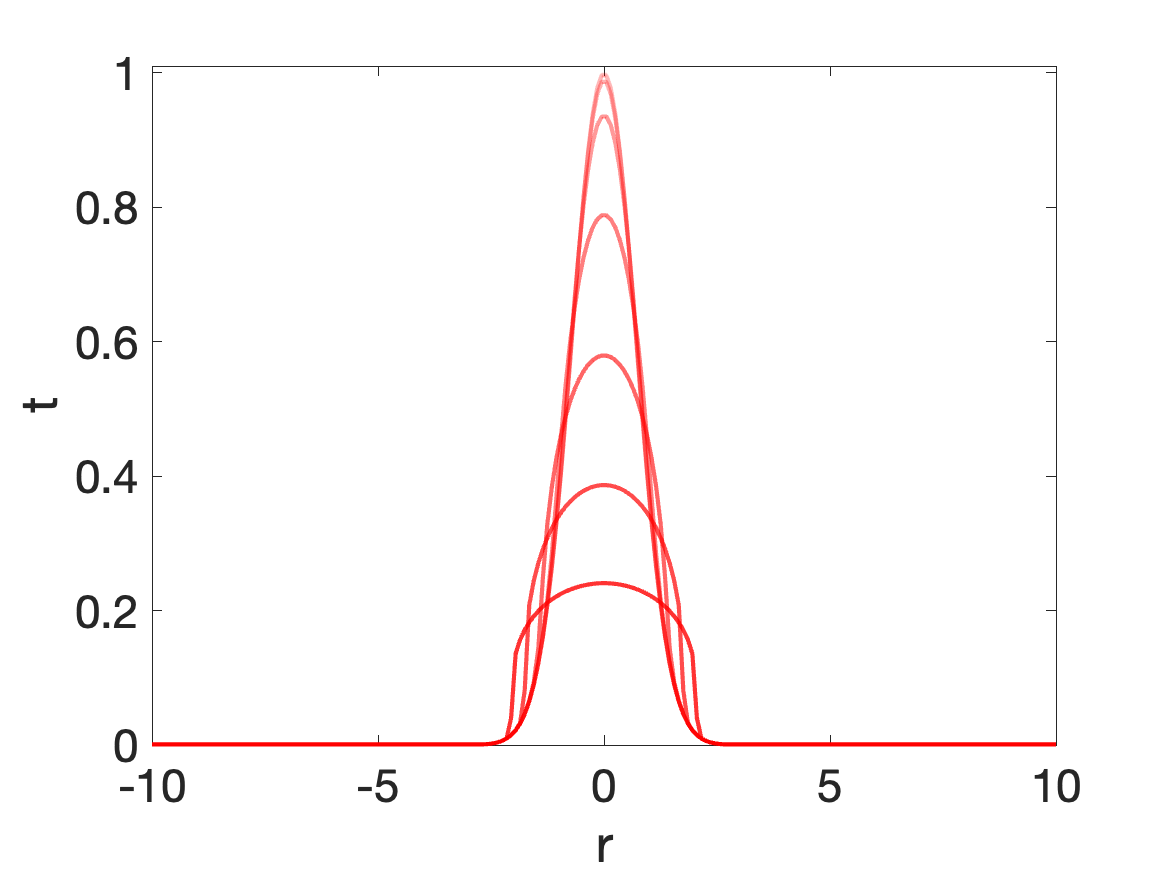} \\
\caption{Snapshots of numerical solutions of  (\ref{biofilm_orig}) for 
(a)  $R(t)/R_0 = R_s(t)$, (b) $R(t)/R_0 = e^t$, (c) $R(t)/R_0 = e^{-t}$
taken at times $1,2,3,4,5,6,7,8$, starting from $e^{-r^2}$ at $t=0$.
(d)-(e)
Associated scaled height profiles $\tilde h$ after division by $e^t$. 
The line color darkens slightly as $t$ grows.
Parameters: $\delta =0$,  $K=10^{-5}$, $h_{\rm inf}= 10^{-3}$,
$dr= 10^{-1}$ and $dt = 5 \cdot 10^{-4}, 2 \cdot 10^{-4}, 5 \cdot 10^{-4}$, 
respectively.}
\label{bump2}
\end{figure}

\begin{figure}[!hbt]
\hskip 0.5cm (a) \hskip 4cm (b) \hskip 4cm (c)\\
\includegraphics[width=5cm]{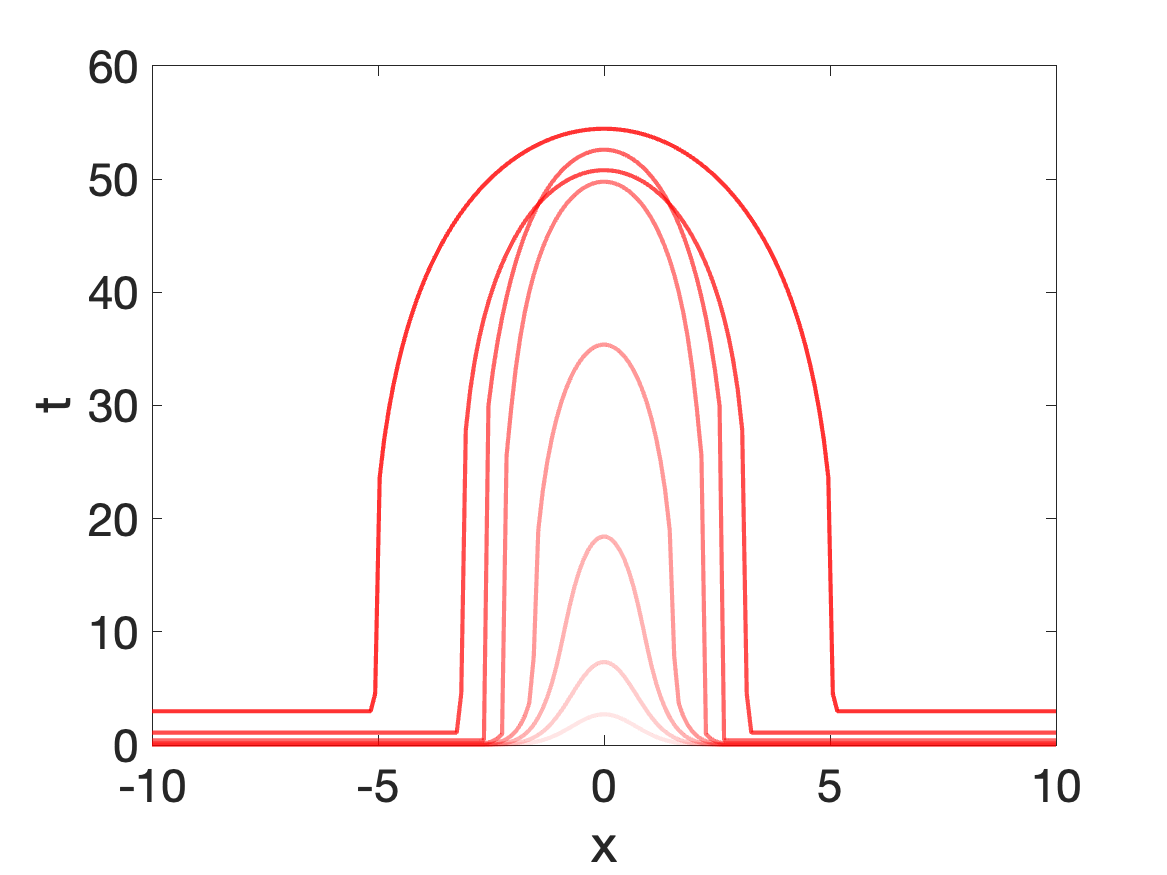}
\includegraphics[width=5cm]{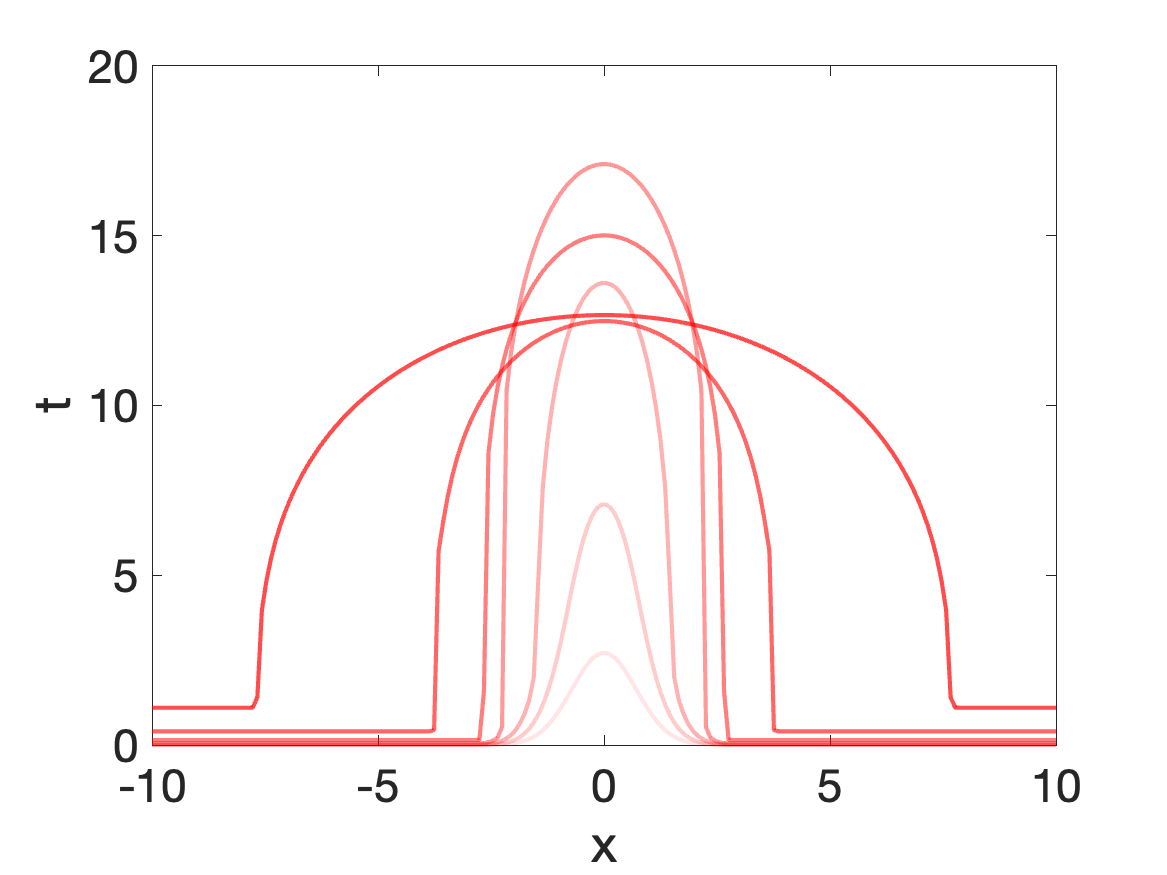}
\includegraphics[width=5cm]{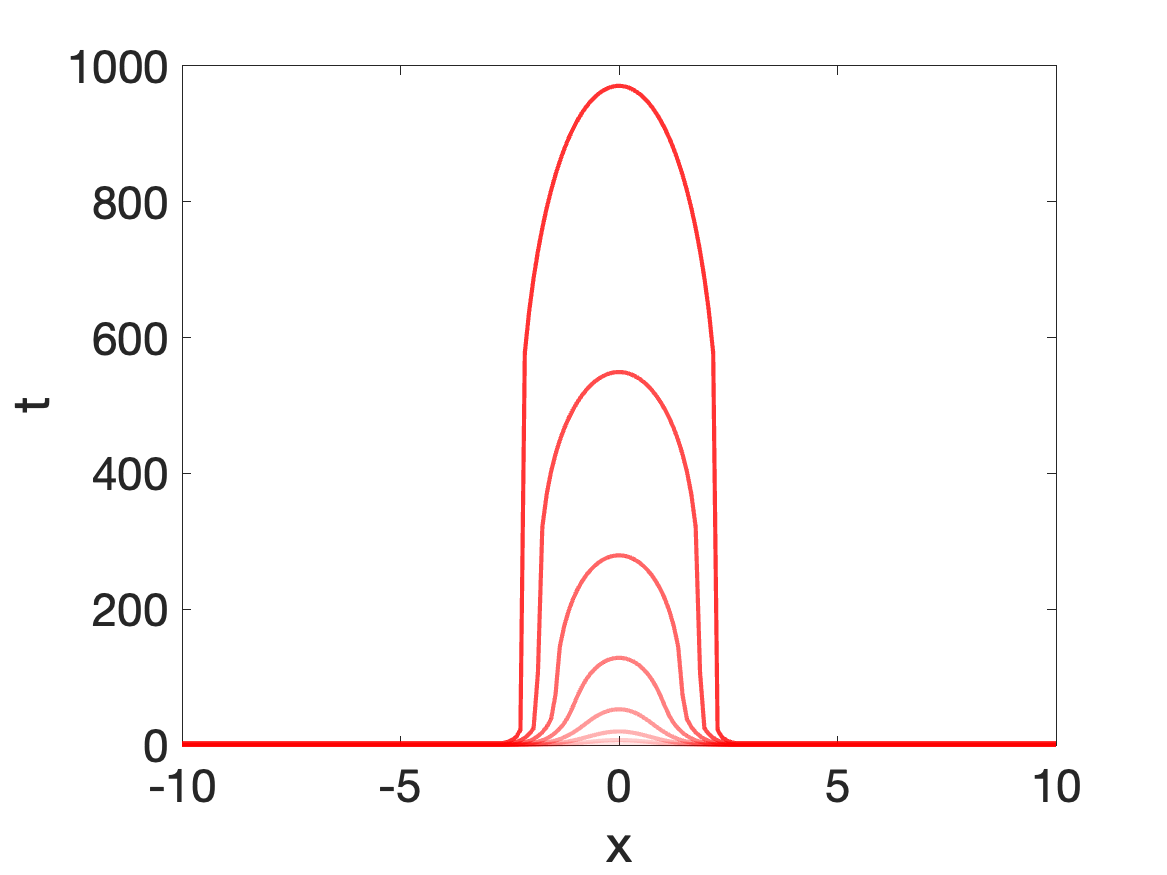} \\
\hskip 0.5cm (d) \hskip 4cm (e) \hskip 4cm (f)\\
\includegraphics[width=5cm]{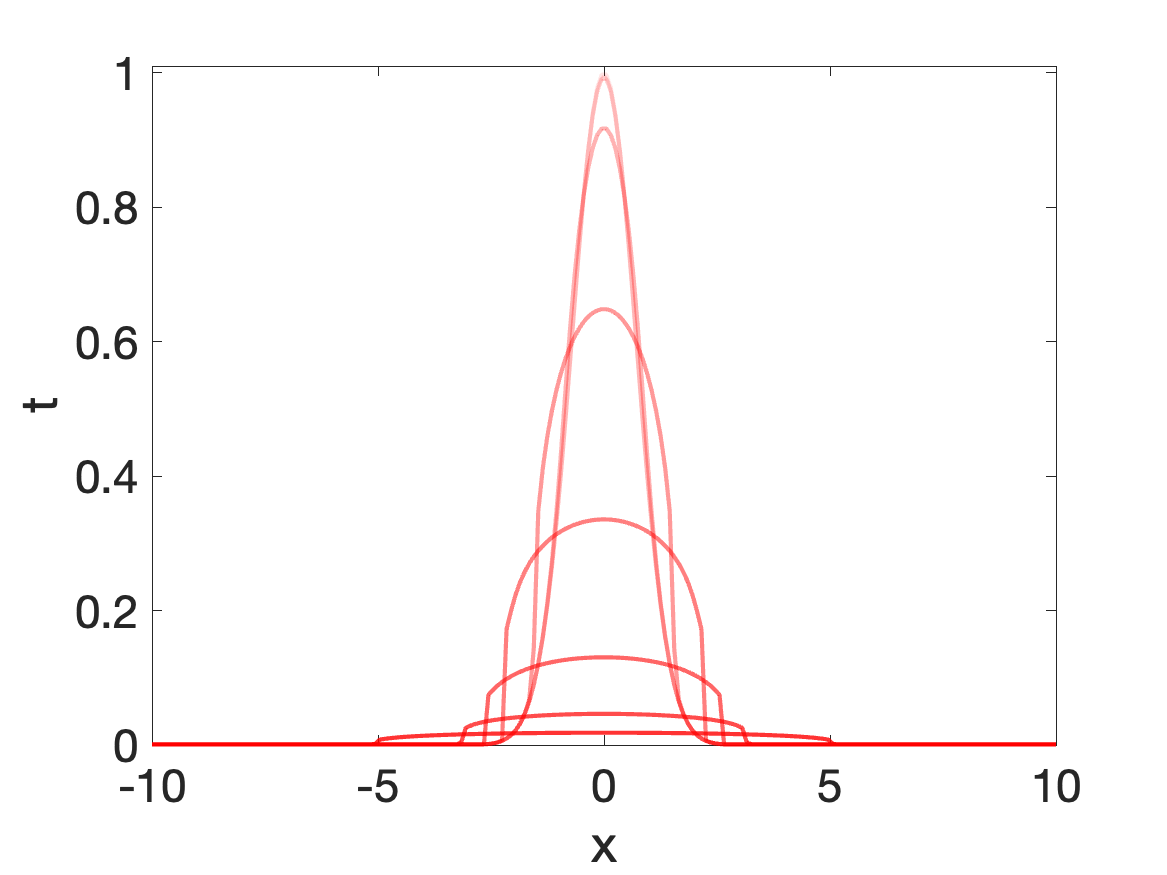}
\includegraphics[width=5cm]{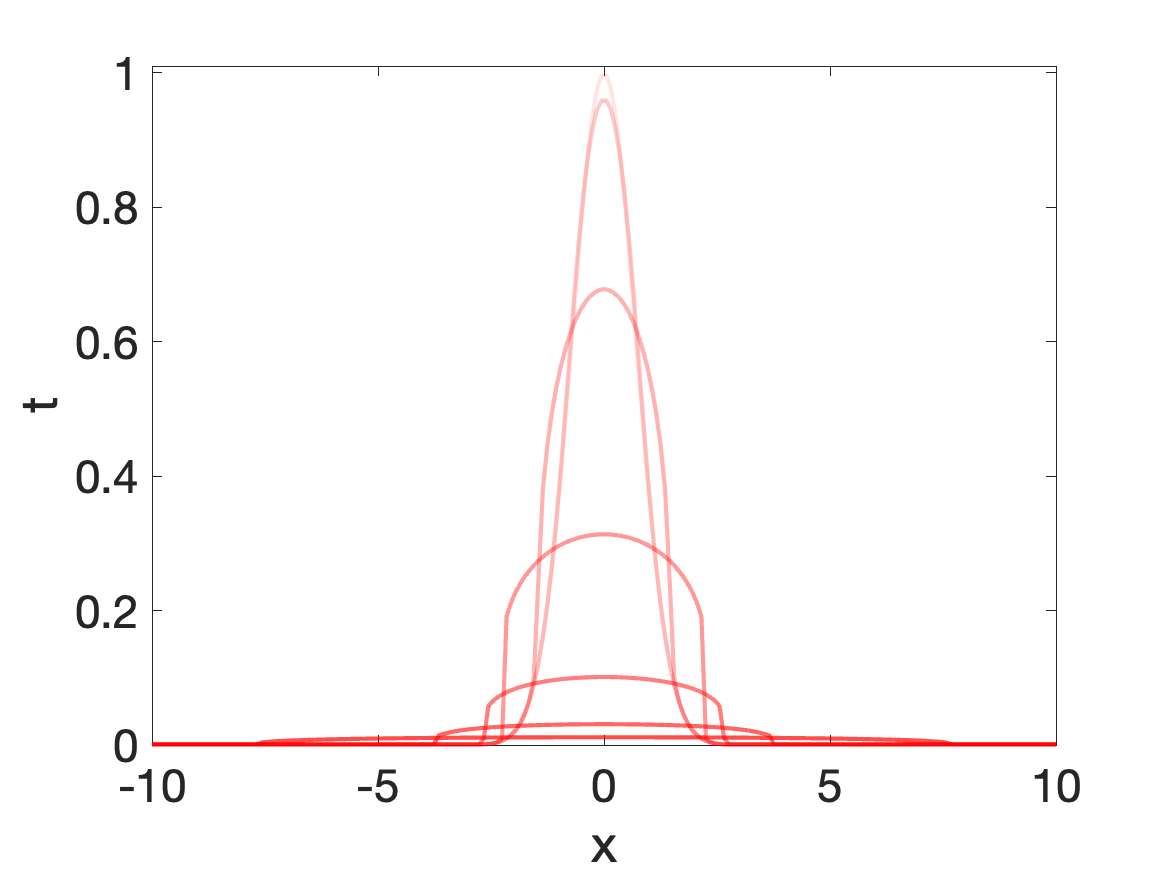}
\includegraphics[width=5cm]{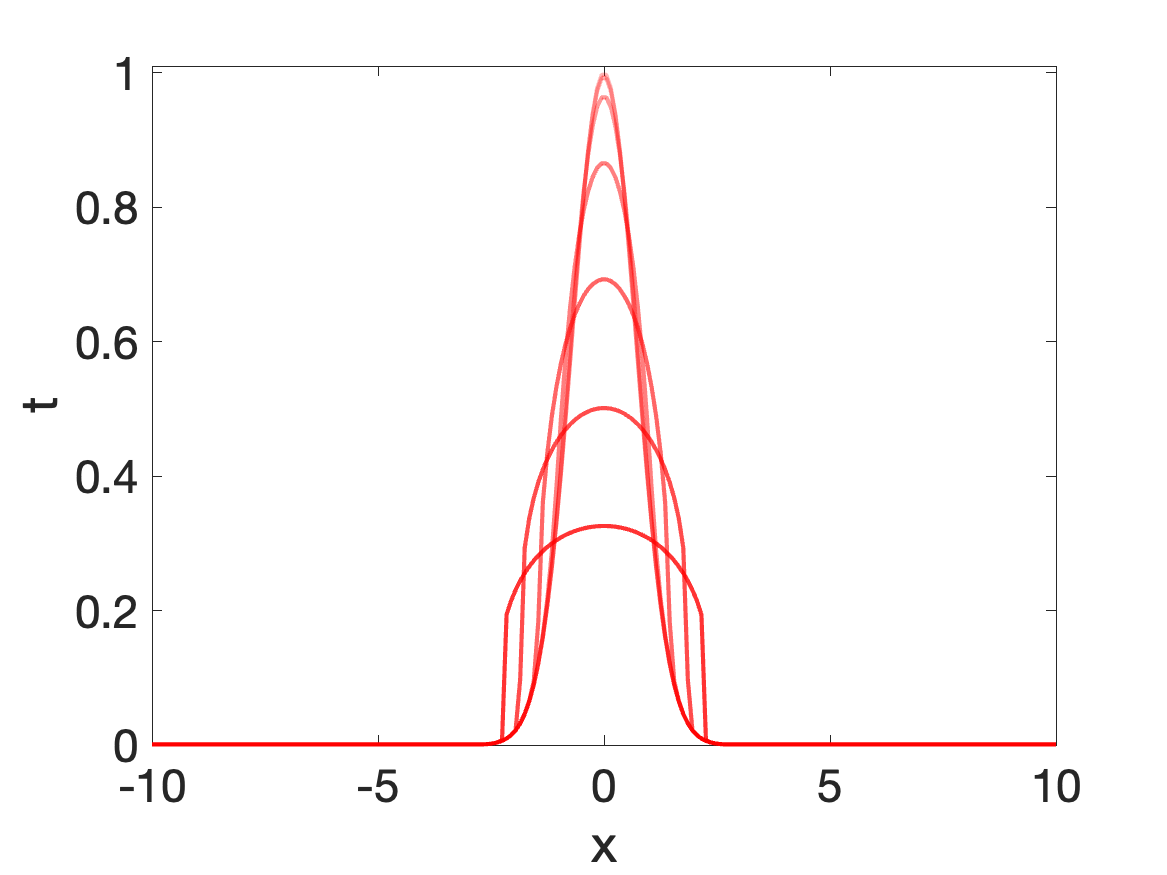} \\
\caption{
Counterpart of Figure \ref{bump2}
for model (\ref{biofilm_orig_x}).
We keep the same choices except the time steps
$dt = 5 \cdot 10^{-4}, 1.25 \cdot 10^{-4}, 5 \cdot 10^{-4}$,
for (a)-(d), (b)-(e), (c)-(f), respectively.
}
\label{bump3}
\end{figure}

Models (\ref{biofilm}) and  (\ref{biofilm_x}) also admit any constant as a solution.
For the same three choices of $R(t)$ considered in Figures \ref{bump2} 
and \ref{bump3},
constant profiles and wavy perturbations of these profiles remained stable during the
simulation time, as expected from the theoretical results established in this paper
for the non-autonomous $f$-model.


The autonomous $f-$model for $R(t) = e^{-3 t}$ has features that might lead to instability, as it follows
from a linear stability analysis about constant  profiles in the original model (\ref{biofilm_x}). Let us linealize (\ref{biofilm_x})  about a constant $c$: $h= c + \hat h$.  Neglecting quadratic terms, we find
\[
\hat h_t - {K \over R_0} c^3 \hat h_{xxt} + {K\over 2 R_0} c^3 \hat h_{xx} - \delta^2 e^{2t }c^2 h_{xx} = 0.
\]
We seek plane wave solutions $\hat h = e^{\imath \alpha x + \beta t}$. Inserting this expression in
the linearized equation we obtain
\begin{eqnarray*}
\beta + {c^3 K \over R_0} \alpha^2 \beta - {c^3 K\over 2 R_0} \alpha^2 + \delta^2 c^2 \alpha ^2 = 0  
\implies  \beta =  {{c^3 K\over 2 R_0} \alpha^2 - \delta^2 e^{2t } c^2 \alpha ^2 \over
1 + {c^3 K \over R_0} \alpha^2 }.
\end{eqnarray*}

We distinguish three cases depending on the value of $\delta$. If $\delta = 0$, the real part of $\beta$ is always positive and constant profiles are unstable. In this case, there is a possibility of film rupture in finite time. On the other hand, if $\delta > 0$, such unstable solutions do not appear. Furthermore, if $\delta$ decays in time, for instance as $\delta(t) = e^{-t}$, then some small constant profiles may become stable over time.

\begin{figure}[!hbt]
\includegraphics[scale=0.42]{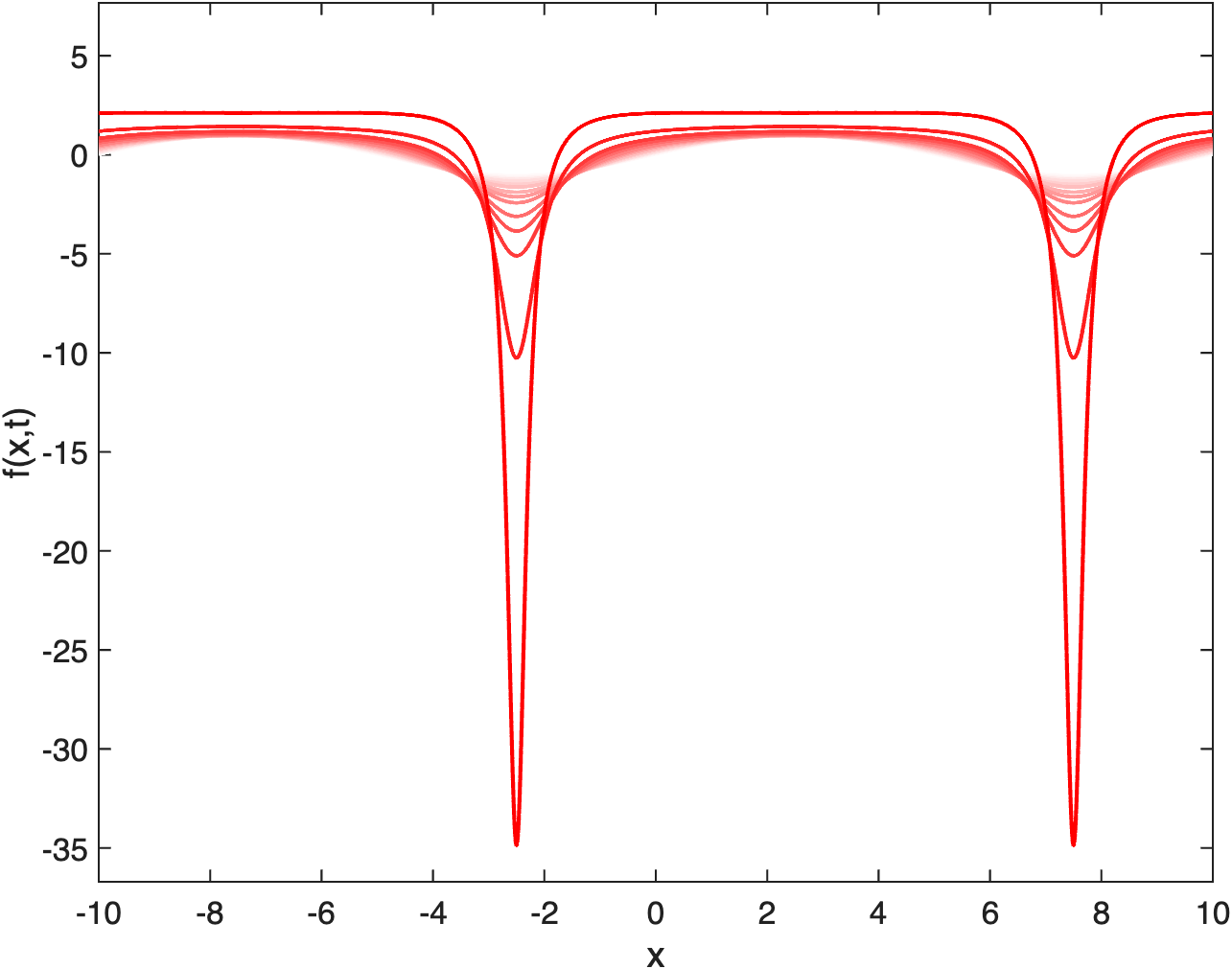}
\caption{Snapshots of the solution $f(x,t)$ of \eqref{model:aut}
at different times, illustrating the temporal evolution of the model. Initially, the solution is smooth and close to the initial condition. As time progresses, sharp negative peaks emerge symmetrically, growing rapidly in amplitude and suggesting the development of a finite-time singularity.
}
\label{singular}
\end{figure}

To support the instability observed in the case $\delta = 0$, we performed numerical simulations of~\eqref{model:aut}. We used a Fourier collocation scheme to discretize the spatial domain, and integrated in time using the classical adaptive Runge--Kutta method with adaptative step of order 4 and 5. In particular, with the initial data $f(x,0) = \sin(2x)$
and $N = 2^{14}$ spatial nodes, we observe that the system appears to be unstable, in agreement with the theoretical prediction for $\delta = 0$, see Figure \ref{singular}. Let us emphasize that the initial interface is $1+ \varepsilon f$. In particular, when $f(x)<-\frac{1}{\varepsilon}$ the asymptotic models lacks physical meaning. Although this behavior suggests the possibility of a finite time singularity at $t=0.35$, the numerics are not strong enough to claim that there is a numerical evidence of such behavior.


\end{document}